\documentclass  [article]{amsart}

\usepackage{amsmath, amssymb, amsfonts, latexsym}

\usepackage{amsfonts}
\usepackage{amsmath}

\usepackage{amssymb}
\newtheorem{acknowledgement}{Acknowledgement}

\newtheorem{lemma}{Lemma}[section]
\newtheorem{Add in proof}{Add in proof}[section]
\newtheorem{prop}{Proposition}[section]
\newtheorem{defn}{Definition}[section]
\newtheorem{remark}{Remark}[section]
\newtheorem{theorem}{Theorem}

\def\R{{\Bbb R}}

\def\laplacian{\bigtriangleup}

 \def\a{\alpha}
\def\intave#1{-\kern-10.7pt\int_{\,#1}}
\def\b{\beta}

\def\g{\gamma}

\def\<{\langle}
\def\>{\rangle}
\def\({\left(}
\def\){\right)}

\def\limsup{\operatornamewithlimits{lim\,sup}}
\def\intave#1{-\kern-10.7pt\int_{\,#1}}

\begin{document}
\title{Blow-up criteria of strong solutions to
the  Ericksen-Leslie system  in $\Bbb R^3$}

\numberwithin{equation}{section}
\author{Min-Chun Hong, Jinkai Li and Zhouping Xin}

\address{Min-Chun Hong, Department of Mathematics, The University of Queensland\\
Brisbane, QLD 4072, Australia}

\address{Jinkai Li, The Institute of Mathematical Sciences,
The Chinese University of Hong Kong,  Shatin, N.T., Hong Kong}

\address{Zhouping Xin, The Institute of Mathematical Sciences,
The Chinese University of Hong Kong,  Shatin, N.T., Hong Kong}

\begin{abstract} In this paper, we establish
the local well-posedness and  blow-up criteria of
strong solutions to the Ericksen-Leslie system in $\Bbb R^3$ for
the well-known Oseen-Frank model. The local existence of
strong solutions to liquid crystal flows is obtained by using the
Ginzburg-Landau approximation approach to guarantee the
constraint that the direction vector of the fluid is of length one.
We establish four kinds of blow-up criteria, including (i) the Serrin type; (ii) the
Beal-Kato-Majda type; (iii) the mixed type, i.e., Serrin type condition for one field and Beal-Kato-Majda type condition on the other one; (iv) a new one, which characterizes the maximal existence time of the strong solutions to the Ericksen-Leslie system in terms of Serrin type norms of the strong solutions to the Ginzburg-Landau approximate system. Furthermore, we also prove that the strong solutions of the Ginzburg-Landau approximate system converge to the strong solution of the Ericksen-Leslie system up to the maximal existence time.
     \end{abstract}

 \maketitle

\pagestyle{myheadings} \markright {The Ericksen-Leslie system}

\section{\bf Introduction}

The Ericksen-Leslie theory is successful
 in describing dynamic flows of liquid crystals in physics,  which is based on the fundamental Oseen-Frank
model.
Mathematically,  the static theory of nematic liquid crystals involves
a unit vector field $u$ in a region $\Omega\subset\R^3$. The
Oseen-Frank density $W (u,\nabla u)$ is given by
\begin{equation*}
W(u,\nabla u)=k_1(\text{div } u)^2+k_2 (u\cdot\text{curl }u)^2+k_3
|u\times \text{curl } u|^2+k_4[\text {tr} (\nabla u)^2-(\text {div
}u)^2],
\end{equation*}
where $k_1$, $k_2$, $k_3$ and $k_4$ are positive constants.
The free energy for a configuration $u\in H^1(\Omega ; S^2)$ is
\begin{equation*}
    E(u; \Omega )=\int_{\Omega} W(u, \nabla u)\,dx.
\end{equation*}
The Euler-Lagrange system for the Oseen-Frank energy
$E(u,\Omega)$ is:
\begin{align*}
 &\quad \nabla_{\a} \left [W_{p_{\a}^i} (u,\nabla u)- u^lu^i
W_{p_{\a}^l}(u,\nabla u) \right ]-W_{u^i}(u,\nabla
u)+W_{u^l}(u,\nabla u)u^l u^i \\
& +W_{p_{\a}^l}(u,\nabla u)\nabla_{\a}u^l u^i
+W_{p_{\a}^l}(u,\nabla u) u^l\nabla_{\a}u^i=0\quad \text {in
}\Omega \nonumber
  \end{align*}
 for $i=1,2,3$   (see \cite {HX}), where the standard summation convention is adopted.

Since the divergence of $\text {tr} (\nabla u)^2-(\text {div
}u)^2$ is free (\cite {GMS}), one can rewrite the density
$W(u,\nabla u)$ as
\begin{equation*}
W(u,\nabla u) = a|\nabla u|^2 +
V(u, \nabla u),\quad a =\min\{k_1, k_2,k_3\}>0,
\end{equation*}
where
$$V(u, \nabla u)=(k_1-a)(\text{div }
u)^2+(k_2-a) (u\cdot\text{curl }u)^2+(k_3-a) |u\times \text{curl }
u|^2. $$

Hardt, Kinderlehrer and Lin  in \cite {HKL1}  proved  that a
minimizer $u$ of the energy $E$ is smooth away from a closed set $\Sigma$ of
$\Omega$. Moreover, $\Sigma$ has Hausdorff dimension strictly less
than one.  See further contributions in \cite {GMS} and  \cite{Ho1} about the static theory of liquid crystals.

Dynamic motion of liquid
crystals are described by the Ericksen-Leslie system, including the
velocity vector $v$ and the direction vector $u$ of the fluid (see
\cite {Er} and \cite {Le}).  More precisely,  let
$v=(v^1,v^2,v^3)$ be the velocity vector of the fluid and
$u=(u^1,u^2,u^3)$ the unit direction vector. The Ericksen-Leslie
system is given by (e.g. \cite {LL1}
and \cite {LL2})
\begin{equation}
v^i_t+(v\cdot\nabla )v^i-\laplacian v^i+\nabla_{x_i} p=-
\nabla_{x_j}(\nabla_{x_i}u^kW_{p_j^k}(u, \nabla u)),  \label{1.5}
 \end{equation}
\begin{equation} \nabla \cdot v=0,  \label{1.6}
 \end{equation}
\begin{align} \label{1.7}
 u^i_t+(v\cdot \nabla  )u^i =&  \nabla_{\a} \left [W_{p_{\a}^i}
(u,\nabla u)- u^ku^i V_{p_{\a}^k}(u,\nabla u) \right
]-W_{u^i}(u,\nabla
u)\\
&+W_{u^k}(u,\nabla u)u^k u^i+W_{p_{\a}^l}(u,\nabla
u)\nabla_{\a}u^l u^i +V_{p_{\a}^k}(u,\nabla u) u^k\nabla_{\a}u^i\nonumber
  \end{align}
for $i=1,2,3$.
Here $\nu$, $\lambda$ are given positive
constants, and $p$ is the pressure.

The system (\ref{1.5})--(\ref{1.7}) is a system of the
Navier-Stokes equations coupled with the gradient flow for the
Oseen-Frank model, which is an extension of the harmonic map flow (\cite {ES}).  Caffarelli, Kohn and  Nirenberg \cite {CKN} established the fundamental result  on the existence and partial
regularity  of the global modified weak solutions of the Navier-Stokes equations (See also \cite{L}, \cite {TX}).
On the other hand,  Struwe  \cite{St} and  Chen-Struwe \cite{CS} established the existence and partial regularity of  global weak solutions of the harmonic map flow between manifolds.
There
is an interesting question to establish the global existence of
weak solutions of (\ref{1.5})-(\ref{1.7}) supplemented with initial or initial-boundary conditions.
The question for the case of $k_1=k_2=k_3$ was answered by
the first author in \cite {Ho3} in $\R^2$  and Lin-Lin-Wang \cite
{LLW} in a bounded domain of $\R^2$ independently. Recently,
the first and third authors \cite {HX} proved  the global existence of weak solutions of
the general Ericksen-Leslie system (\ref{1.5})--(\ref{1.7}) in $\R^2$. However, the
question on the global weak solution on the system in $3D$ is
still unknown. In the study of the Navier-Stokes
equations, there are two well-known blow-up
criteria for the strong (smooth) solutions: the Serrin (also called Ladyzhenskaya-Prodi-Serrin type) criterion \cite{Se} and the
Beal-Kato-Majda type criteria \cite{BKM}. Recently, for the simplified model, i.e.
$k_1=k_2=k_3$, the local strong solutions was obtained by Wen and
Ding \cite{WD}, and the blow up criterions were obtained by Huang and
Wang \cite{HW}, and there have been many new results developed in this direction \cite{HWW}.

In this paper, we consider the Cauchy problem to the Ericksen-Leslie system (\ref{1.5})--(\ref{1.7})
for the general Oseen-Frank model in $\R^3$. Suppose that the
initial data is given by
\begin{equation}\label{1.9}
u(x, 0)=u_0(x),\qquad v(x, 0)=v_0(x).
\end{equation}
Throughout this paper, we always assume that $(u_0, v_0)$ satisfies
\begin{eqnarray*}
&v_0\in H^1(\mathbb R^3),\quad\textmd{div}\,v_0=0,\quad u_0-b\in
H^2(\mathbb R^3),\quad|u_0|=1
\end{eqnarray*}
for some constant unit vector $b$.

In order to state our results, we give the definition of strong solutions and introduce some notations.

\begin{defn}
For any $T>0$, a couple $(u, v)$ is called a strong solution to the system (\ref{1.5})--(\ref{1.7}) in $\mathbb R^3\times(0, T)$ if and only if
\begin{eqnarray*}
&&u\in L^2(0, T; H_b^3(\mathbb R^3)),\quad\partial_tu\in L^2(0, T; H^1(\mathbb R^3)), \quad|u|=1,\\
&&v\in L^2(0, T; H_\sigma^2(\mathbb R^3)),\quad\partial_tv\in L^2(0, T; L^2(\mathbb R^3))
\end{eqnarray*}
and it satisfies the equation (\ref{1.5})--(\ref{1.7}) a.e. $(x, t)\in\mathbb R^3\times(0, T)$.
\end{defn}

\begin{defn}
A finite time $T^*>0$ is called the maximal existence time of a strong solution $(u, v)$ to the system (\ref{1.5})--(\ref{1.7}) if and only if $(u, v)$ is a strong solution in $\mathbb R^3\times(0, T)$ for all $T<T^*$ and
$$
\lim_{T\nearrow T^*}\|(\nabla u, v)\|_{L^2(0, T; H^2(\mathbb R^3))}=\infty.
$$
\end{defn}

The  maximal existence time  of the strong solution  to the approximate system (\ref{1.10})--(\ref{1.12}) can be defined similarly.
For $T>0$, we denote
\begin{align*}
&J_1(T)=\inf_{(q, r)\in\mathcal O}\|\nabla u\|_{L^q(\frac{T}{2}, T; L^r(\mathbb R^3))}+\inf_{(q, r)\in\mathcal O}\|v\|_{L^q(\frac{T}{2}, T; L^r(\mathbb R^3))},\\
&J_2(T)=\|\omega\|_{L^1(\frac{T}{2}, T; BMO(\mathbb R^3))}+\|\Delta u\|_{L^1(\frac{T}{2}, T; L^\infty(\mathbb R^3))},\\
&J_3(T)=\inf_{(q, r)\in\mathcal O}\|v\|_{L^q(\frac{T}{2}, T; L^r(\mathbb R^3))}+\|\Delta u\|_{L^1(\frac{T}{2}, T; L^\infty(\mathbb R^3))},\\
&J_4(T)=\|\omega\|_{L^1(\frac{T}{2}, T; BMO(\mathbb R^3))}+\inf_{(q, r)\in\mathcal O}\|\nabla u\|_{L^q(\frac{T}{2}, T; L^r(\mathbb R^3))},
\end{align*}
where $\omega=\nabla\times v$ and
$$
\mathcal O=\left\{(q, r)\in\mathbb R^2\left|\frac{2}{q}+\frac{3}{r}=1, q\in[2,\infty), r\in(3,\infty]\right.\right\}.
$$

Then, we have the following results on the local existence and blow up criteria of strong solutions to the system (\ref{1.5})--(\ref{1.9}).

\begin{theorem} \label{thm1}
The system (\ref{1.5})--(\ref{1.9})
has a unique strong solution $(u, v)$ in $\mathbb R^3\times(0, T^*)$ for some positive number $T^*$ depending only on the initial data. The maximal existence time $T^*<\infty$ can be described as
$$
J_1(T^*)=J_2(T^*)=J_3(T^*)=J_4(T^*)=\infty.
$$
Moreover, for any $T>0$, $J_1(T), J_2(T), J_3(T)$ and $J_4(T)$ are equivalent in the
following sense:
\begin{eqnarray*}
J_1(T)=\infty \Longleftrightarrow J_2(T)=\infty \Longleftrightarrow
J_3(T)=\infty \Longleftrightarrow J_4(T)=\infty.
\end{eqnarray*}
\end{theorem}

The proof of Theorem \ref {thm1} is divided into two parts: local existence and blow-up criterion of the strong solution.
 For the proof of the local existence  of  the Ericksen-Leslie system,  the
main difficulty  is that  the  system (1.4)--(1.6) is not a standard parabolic system in the sense described in \cite {LSU} or \cite {Ei}. As a result, the
 constraint $|u|=1$ cannot be derived directly from the system by using the maximum principle. To overcome this difficulty, we follow  the same idea in \cite
{HX} to consider the approximating Ericksen-Leslie system
 in the following:
\begin{equation}v^i_t+(v\cdot\nabla )v^i-\laplacian v^i+\nabla_{x_i} p=
-\nabla_{x_j}(\nabla_{x_i}u^kW_{p_j^k}(u,\nabla u)),
\quad \label{1.10}
\end{equation}
\begin{equation} \nabla \cdot v=0,
\label{1.11}
\end{equation}
\begin{equation}
u^i_t+(v\cdot \nabla  )u^i =   \nabla_{\a} \left [W_{p_{\a}^i} (u,
\nabla u) \right ]-W_{u^i}(u,\nabla u)+\frac 1 {\varepsilon^2} u^i
(1-|u|^2) \label{1.12}
\end{equation}
for  $i=1,2,3$,  prescribing the initial condition
(\ref{1.9}). However, it should be noted that the condition that $ u_0\in H^1_b$ and $v_0\in L^2$ is insufficient to establish the  local existence of
the Ericksen-Leslie system in 3D. Instead we must assume that $ u_0\in H^2_b$ and $v_0\in H^1$. Under this condition, we can establish  uniform estimates in $\varepsilon$  on higher derivatives  of solutions $(u_{\varepsilon}, v_{\varepsilon})$ to the approximation system
(\ref{1.10})--(\ref{1.12}) in a short time and prove the local existence. In order to obtain such uniform estimates of $\int |\nabla^2 u_{\varepsilon}|^2+|\nabla v_{\varepsilon}|^2 +|\partial_t u_{\varepsilon}|^2$, the first key idea is to prove  that $|u_\varepsilon|$ is close to 1 as $\varepsilon$ goes to zero and the second key idea is  to control a difficult term term $\int\frac{1-|u_\varepsilon|^2}{\varepsilon^2}|\partial_t u_\varepsilon|^2$ by using the decomposition
$$
\partial_tu_\varepsilon=\frac{1}{|u_\varepsilon|^2}(\partial_tu_\varepsilon\cdot u_\varepsilon)u_\varepsilon-\frac{1}{|u_\varepsilon|^2}(\partial_tu_\varepsilon\times u_\varepsilon)\times u_\varepsilon.
$$
We note that $\partial_tu_\varepsilon\times u_\varepsilon$ is independent of $\varepsilon$ by equation (\ref{1.12}).

To establish the blow up criteria of the Ericksen-Leslie system, we need a prior estimates on high derivatives of the solution before the maximal existence time $T^*$. Two kinds of estimates are established, which roughly speaking involve the $L^\infty(H^1)$ and $L^\infty(H^2)$ bounds of $(v,\nabla u)$, respectively. One of the key ideas in establishing such estimates is using the constraint $|u|=1$ to handle the terms like $u\cdot\Delta^2u$ by reducing the order of the derivatives. In Theorem 1, we impose a Serrin type condition or Beal-Kato-Majda condition on $u$ or $v$. If we impose a Serrin type condition on the velocity field $v$, the $L^\infty(H^1)$-bounds on $(v,\nabla u)$ is sufficient for the proof, no matter what kind of condition is imposed on the direction field $u$. If imposing a Beal-Kato-Majda condition on the velocity field $v$, we have to analysis the  second kind estimate $L^\infty(H^2)$. In this case, a new logarithmic Sobolev type inequality is needed to control the $L^1([0,T];L^\infty (\R^3))$ norm of $\nabla v$ in term of its $L^1([0,T]; BMO(\R^3))$ and the norms of higher order derivatives.

\begin{remark} (i) $J_1(T^*)=\infty$ is a Serrin type condition  for both fields $u$ and $v$;
$J_2(T^*)=\infty$ is  a Beal-Kato-Majda type condition  for both fields; $J_3(T^*)=\infty$ and $J_4(T^*)=\infty$
are a Serin type condition for one field and a Beal-Kato-Majda type for the other one.

(ii) Recently, Huang-Wang \cite {HW} established the blow up
criterion  of the form
$$
\|\omega\|_{L^1_t(L^\infty_x)}+\|\nabla
u\|_{L^2_t(L^\infty_x)}=\infty,
$$
for the simplified model, which is a special case of $J_4$ in
Theorem \ref{thm1}.

(ii)
Theorem \ref{thm1} shows that the Serrin type condition is equivalent to the Beal-Kato-Majda type
in our case.
\end{remark}

 By comparing with the well-known result of Chen-Struwe \cite{CS} on the harmonic map flow, it is of interests to investigate the convergence  problem of solutions of the approximating system (\ref{1.10})--(\ref{1.12}). In fact, the approximating Ericksen-Leslie system (\ref{1.10})--(\ref{1.12}) was first introduced by Lin-Liu in \cite {LL1} through the Ginzburg-Landau approximation. They proved global existence of the
classical solution of the approximate system (\ref{1.10})--(\ref{1.12}) with
(\ref{1.9}) in dimension two and the weak solution
of the same system in dimension three for the case of
$k_1=k_2=k_3$. Since their estimates depends on the parameter $\varepsilon$ (also see \cite {LL2}), it is unknown whether as $\varepsilon\to 0$ the
 solutions $(u_{\varepsilon},v_{\varepsilon})$ of
(\ref{1.10})--(\ref{1.12}) converge to the solution of the
original Ericksen-Leslie system (\ref{1.5})--({\ref{1.7}). In this
paper, we can answered this problem and prove that  these strong solutions $(u_{\varepsilon},v_{\varepsilon})$ of the approximate system (\ref{1.10})--(\ref{1.12}) converge to the strong solution $(u, v)$ of the original Ericksen-Leslie system up to the maximal existence time of $(u, v)$. More precisely,  we have:

\begin{theorem}\label{thm2}
Let $(u, v)$ be a strong solution to the system (\ref{1.5})--(\ref{1.9}) in $\mathbb R^3\times(0, T^*)$. Let $(u_\varepsilon, v_\varepsilon)$ be the unique strong solution to the system (\ref{1.10})--(\ref{1.12}) in $\mathbb R^3\times(0, T_\varepsilon^*)$  with (\ref{1.9}), where $T_\varepsilon^*$ is the maximal existence time of  (\ref{1.10})--(\ref{1.12}).
Then for sufficiently small $\varepsilon$, $T_\varepsilon^*\geq T^*$ and
for any $T\in(0, T^*)$, it holds that
$$
(\nabla u_\varepsilon, v_\varepsilon)\rightarrow(\nabla u, v),\qquad\mbox{ in }{L^\infty(0, T; L^2(\mathbb R^3))\cap L^2(0,T; H^1(\mathbb R^3))}
$$
and
$$
\varlimsup_{\varepsilon\rightarrow0}(\|(\nabla u_\varepsilon, v_\varepsilon)\|_{L^\infty(0, T; H^1(\mathbb R^3))}+\|(\nabla u_\varepsilon, v_\varepsilon)\|_{W^{2,1}_2(\mathbb R^3\times(0, T))})<\infty,
$$
where $\|f\|_{W^{2,1}_2(\mathbb R^3\times(0, T))}=\|f\|_{L^2(0,T; H^2(\mathbb R^3))}+\|\partial_tf\|_{L^2(\mathbb R^3\times(0, T))}$.

Furthermore, $T^*<\infty$ is the maximal existence time if and only if
$$
\lim_{\varepsilon\rightarrow0}\|(\nabla u_\varepsilon, v_\varepsilon)\|_{L^q(0, T^*; L^r(\mathbb R^3))}=\infty
$$
for any $(q, r)\in \mathcal O$, with $\mathcal O$ being the same set stated as before.
\end{theorem}

The key in the proof of  Theorem \ref{thm2} is to establish the
strong convergence and uniform estimates, which is divided in three steps: in step 1, we prove the strong convergence and uniform estimates up to a time $T_M$, where $M$ is a constant depending only on the initial data and $T$; in step 2, we show that if the strong convergence and uniform estimate hold true up to $T_1$ with $T_1<T$, then they hold true up to another time $T_2:=\min\{T, T_1+T_M\}$; in step 3, we prove the strong convergence and uniform estimate up to time $T$. To prove the strong convergence up to $T_M$, we need to derive high order estimates up to time $T_M$ and prove that the energy of $(u_\varepsilon, v_\varepsilon)$ is small outside a big ball uniformly for $\varepsilon$. High order estimates of these strong solutions are guaranteed by Proposition \ref{prop}, which, roughly speaking, states that the existence time and the uniform estimates of these strong solutions depend only on the $H^1$ bounds of the initial data $(\nabla u_{\varepsilon}(0), v_{\varepsilon}(0))$ and the $L^2$ bounds of $\frac{1-|u_{\varepsilon}(0)|^2}{\varepsilon^2}u_{\varepsilon}(0)$, while the uniform smallness outsider a big ball can be guaranteed by our Lemma \ref{lem4.2}, which is a local type of energy inequality. Using these two tools, we can prove the strong convergence of these solutions up to the time $T_M$. If the strong convergence and uniform estimate hold true up to time $T_1$ for some $T_1<T$, by the aid of the strong convergence and the uniform estimates up to time $T_1$, we show that the $H^1$ bounds of $(\nabla u_\varepsilon(T_1), v_\varepsilon(T_1))$ and the $L^2$ bounds of $\frac{1-|u_\varepsilon(T_1)|^2}{\varepsilon^2}u_{\varepsilon}(T_1)$ is controlled by $M$. As a result, starting from $T_1$ and taking $(u_\varepsilon(T_1), v_{\varepsilon}(T_1))$ as initial data, we obtain high order estimates up to time $T_2=\min\{T, T_1+T_M\}$. With this estimate in hand, using the same argument as in step 1, we can show the strong convergence up to $T_2$. Continuing this procedure, we prove the strong convergence up to $T$, and thus complete the proof of Theorem \ref{thm2}. By the aid of the strong convergence and uniform estimate, we can characterize the maximal existence time in term of the strong solutions to the Ginzburg-Landau system.

\begin{remark}
Theorem \ref{thm2} can be viewed as a blow up criterion of the strong solutions to the Ericksen-Leslie system (\ref{1.5})--(\ref{1.7}) in term of the Serrin type norms of the strong solutions to the Ginzburg-Landau approximation system (\ref{1.10})--(\ref{1.12}). It is a new kind of blow up criterion for the Ericksen-Leslie system even for the simplified case that $k_1=k_2=k_3$.
\end{remark}

The rest of the paper is organized as follows. In Section \ref{sec2}, we prove the local existence part of Theorem \ref{thm1};
the blow-up criteria part of of Theorem \ref{thm1} is proved in Section \ref{sec3}; Finally, we give the proof of Theorem \ref{thm2} in
Section \ref{sec4}.

\section{Local existence}\label{sec2}
\allowdisplaybreaks
In this section, we  prove the local existence of strong solutions to the Ericksen-Leslie system by using the Ginzburg-Landau approximation mentioned in Introduction. One can easily check that the following hold
\begin{eqnarray*}
&&W(z,p)\geq a|p|^2,\quad W_{p_\alpha^ip_\beta^j}(z,p)\xi_\alpha^i\xi_\beta^j\geq a|\xi|^2,\quad\forall z\in\mathbb R^3, p,\xi\in\mathbb M^{3\times3},\\
&&|W(u,\nabla u)|\leq C|u|^2|\nabla u|^2,\quad |W_{u^i}(u,\nabla u)|\leq C|u||\nabla u|^2,\\
&&|W_{u^iu^j}(u,\nabla u)|\leq C|\nabla u|^2,\quad |W_{p_\alpha^i}(u,\nabla u)|\leq C|u|^2|\nabla u|,\\
&&|W_{p_\alpha^i p_\beta^j}(u,\nabla u)|\leq C|u|^2,\quad |W_{u^ip_\beta^j}(u,\nabla u)|\leq C|u||\nabla u|.
\end{eqnarray*}
These inequalities will be used in the following text without any further mentions. For the Ginzburg-Landau approximate system (\ref{1.10})--(\ref{1.12}), the following local existence result holds.

\begin{lemma}\label{LLEM2.1}
Suppose that the initial data $(u_{0\varepsilon}, v_{0\varepsilon})$ satisfies
$$
u_{0\varepsilon}-b\in H^2(\mathbb R^3),\quad v_{0\varepsilon}\in H^1(\mathbb R^3),\quad\text{div}~v_{0\varepsilon}=0,
$$
where $b$ is a constant unit vector.
Then there is a positive number $T_\varepsilon^0$, such that the system (\ref{1.10})--(\ref{1.12}) with initial data $(u_{0\varepsilon}, v_{0\varepsilon})$ admits a unique solution $(u_\varepsilon, v_\varepsilon)$ on $\mathbb R^3\times(0, T_\varepsilon^0)$, satisfying
\begin{eqnarray*}
&&u_\varepsilon\in L^2(0, T_\varepsilon^0; H^3_b(\mathbb R^3)),\quad\partial_tu_\varepsilon\in L^2(0,T_\varepsilon^0; H^1(\mathbb R^3)),\\
&&v_\varepsilon\in L^2(0, T_\varepsilon^0; H^2(\mathbb R^3)),\quad\partial_tv_\varepsilon\in L^2(\mathbb R^3\times(0, T_\varepsilon^0)).
\end{eqnarray*}
\end{lemma}

\begin{proof}
We can apply  the standard contraction mapping principle based on the following linearized problem
\begin{equation*}
\tilde v^i_t-\laplacian \tilde v^i+\nabla_{i} P=
-\nabla_{j}(\nabla_{i}u^kW_{p_j^k}(u,\nabla u))-v\nabla v^i,
\end{equation*}
\begin{equation*} \nabla \cdot \tilde v=0,
\end{equation*}
\begin{equation*}
\tilde u^i_t-  \nabla_{\a} \left [W_{p_{\a}^i} (u,
\nabla \tilde u) \right ]=-W_{u^i}(u,\nabla u)-v\nabla  u^i +\frac 1 {\varepsilon^2} u^i
(1-|u|^2)
\end{equation*}
for  $i=1,2,3$. The argument is standard, and thus omitted.
\end{proof}

For strong solutions to the system (\ref{1.10})--(\ref{1.12}), it holds the following basic energy balance.

\begin{lemma}\label{LLEM2.2}
Let $(u_\varepsilon, v_\varepsilon)$ be a strong solution to the system (\ref{1.10})--(\ref{1.12}) in $\mathbb R^3\times (0, T)$.  Then
\begin{align*}
& \frac{d}{dt}\int_{\mathbb R^3}\left[\frac{|v_\varepsilon|^2}{2}+W(u_\varepsilon,\nabla u_\varepsilon)
+\frac{(1-|u_\varepsilon|^2)^2}{4\varepsilon^2}\right]dx\\&
+\int_{\mathbb R^3}(|\nabla v_\varepsilon|^2+|\partial_tu_\varepsilon +(v_\varepsilon\cdot\nabla)u_\varepsilon|^2)\,dt=0
\end{align*}
for any $t\in(0, T)$.
\end{lemma}

\begin{proof}
Multiplying (\ref{1.10}) by $v_\varepsilon^i$ and (\ref{1.12}) by $\partial_tu_\varepsilon^i+v_\varepsilon\cdot\nabla u_\varepsilon^i$ respectively and then summing the resulting equations up and integrating over $\mathbb R^3$, we obtain
\begin{align*}
&\frac{d}{dt}\int_{\mathbb R^3}\left[\frac{|v_\varepsilon|^2}{2}+W(u_\varepsilon,\nabla u_\varepsilon)+\frac{(1-|u_\varepsilon|^2)^2}{4\varepsilon^2}\right]dx\\
&+\int_{\mathbb R^3}(|\nabla v_\varepsilon|^2+|\partial_tu_\varepsilon+(v_\varepsilon\cdot\nabla)u_\varepsilon|^2)dx\\
=&\int_{\mathbb R^3}[\nabla_iu_\varepsilon^kW_{p_j^k}(u_\varepsilon,\nabla u_\varepsilon)\nabla_jv_\varepsilon^i+v_\varepsilon\cdot \nabla u_\varepsilon^i(\nabla_\alpha(W_{p_\alpha^i}(u_\varepsilon,\nabla u_\varepsilon))\\
&-W_{u^i}(u_\varepsilon,\nabla u_\varepsilon))+v_\varepsilon\cdot\nabla u_\varepsilon^i\frac{1}{\varepsilon^2}(1-|u_\varepsilon|^2)u_\varepsilon^i]dx\\
=&\int_{\mathbb R^3}[\nabla_iu_\varepsilon^kW_{p_j^k}(u_\varepsilon,\nabla u_\varepsilon)\nabla_jv_\varepsilon^i-\nabla_\alpha v_\varepsilon\cdot\nabla u_\varepsilon^iW_{p_\alpha^i}(u_\varepsilon,\nabla u_\varepsilon)\\
&-v_\varepsilon^k\nabla_{k\alpha}u_\varepsilon^iW_{p_\alpha^i}(u_\varepsilon,\nabla u_\varepsilon)-v_\varepsilon^k\nabla_ku_\varepsilon^iW_{u_\varepsilon^i}(u_\varepsilon,\nabla u_\varepsilon)]dx\\
=&-\int_{\mathbb R^3} v_\varepsilon^k\nabla_k(W(u_\varepsilon,\nabla u_\varepsilon))dx=0,
\end{align*}
which proves the claim.
\end{proof}

The following high order estimate is one of key lemmas in this paper.

\begin{lemma}\label{lem5.1}
Let $(u_\varepsilon, v_\varepsilon)$ be a strong solution to the system (\ref{1.10})--(\ref{1.12}) in $\mathbb R^3\times(0, T)$. Assume that $\frac{3}{4}\leq|u_\varepsilon|\leq\frac{5}{4}$ on $\mathbb R^3\times(0, T)$. Then for any $t\in(0, T)$, it holds that
\begin{align*}
&\frac{d}{dt}\int_{\mathbb R^3}(|\Delta u_\varepsilon|^2+|\partial_t u_\varepsilon|^2+|\nabla v_\varepsilon|^2)dx+\int_{\mathbb R^3}\left(a|\nabla^3u_\varepsilon|^2+a|\nabla\partial_tu_\varepsilon|^2\right.\nonumber\\
&\left.+|\Delta v_\varepsilon|^2+|\partial_tv_\varepsilon|^2+\frac{1}{\varepsilon^2}|\Delta|u_\varepsilon|^2|^2+\frac{1}{\varepsilon^2}|\partial_t|u_\varepsilon|^2|^2\right)dx\nonumber\\
\leq&C\int_{\mathbb R^3}(|\nabla u_\varepsilon|^2+|v_\varepsilon|^2)(|\nabla^2u_\varepsilon|^2+|\partial_tu_\varepsilon|^2+|\nabla v_\varepsilon|^2)dx,
\end{align*}
where $C$ is a positive constant independent of $\varepsilon$.
\end{lemma}

\begin{proof}
Since $\frac{3}{4}\leq|u_\varepsilon|\leq \frac{5}{4}$ on $\mathbb R^3\times(0, T)$, it follows from (\ref{1.12}) that
\begin{align}
&\left|\frac{1}{\varepsilon^2}(1-|u_\varepsilon|^2)\right|\leq \frac 43 \nonumber\left|\frac{1}{\varepsilon^2}(1-|u_\varepsilon|^2)u_{\varepsilon}\right| \\
\leq &C|\partial_tu_\varepsilon+(v_\varepsilon\cdot\nabla)u_\varepsilon|+C|W_{p_\alpha^ip_\beta^j}(u_\varepsilon,\nabla u_\varepsilon)\nabla_{\alpha\beta}u_\varepsilon^j|\nonumber\\
&+C|W_{u_\varepsilon^jp_\alpha^i}(u_\varepsilon,\nabla u_\varepsilon)\nabla_\alpha u_\varepsilon^j|+C|W_{u^i}(u_\varepsilon,\nabla u_\varepsilon)|\nonumber\\
\leq&C|\partial_tu_\varepsilon+(v_\varepsilon\cdot\nabla)u_\varepsilon|+C|\nabla^2u_\varepsilon|+C|\nabla u_\varepsilon|^2.\label{L.2}
\end{align}

Differentiating (\ref{1.12}) in $x_\beta$, multiplying the resulting equation by $\nabla_\beta\Delta u_\varepsilon^i$ and integrating by parts, one obtains
\begin{align}
&\frac{1}{2}\frac{d}{dt}\int_{\mathbb R^3}|\Delta u_\varepsilon|^2dx-\int_{\mathbb R^3}\nabla_\beta (v_\varepsilon\nabla u_\varepsilon^i)\nabla_\beta\Delta u_\varepsilon^idx\nonumber\\
=&-\int_{\mathbb R^3}[\nabla_{\alpha\beta}(W_{p_\alpha^i}(u_\varepsilon,\nabla u_\varepsilon))-\nabla_\beta(W_{u^i}(u_\varepsilon,\nabla u_\varepsilon))]\nabla_\beta\Delta u_\varepsilon^i dx\nonumber\\
&-\int_{\mathbb R^3}\nabla_\beta\left[\frac{1}{\varepsilon^2}(1-|u_\varepsilon|^2)u_\varepsilon^i\right]\nabla_\beta\Delta u_\varepsilon^i dx. \label{L.2-1}
\end{align}

We will estimate the terms on the right hand side of (\ref{L.2-1}) term by term. Estimates on the first term can be found in \cite{HX}. For completeness, we outline it here. Recalling that $W(u, \nabla u)$ is quadratic in $u^k$ and $\nabla_iu^k$, one has
\begin{align*}
\nabla_{\g \b}W_{p_{\a}^i} (u_\varepsilon,\nabla u_\varepsilon)=& \nabla_{\g} [
W_{u^jp_{\a}^i}(u_\varepsilon,\nabla u_\varepsilon) \nabla_{\b}
u_\varepsilon^j+W_{p_{\a}^i}(u_\varepsilon,\nabla \nabla_{\b}u_\varepsilon) ]\\
=&W_{u^jp_{\a}^i}(u_\varepsilon,\nabla u_\varepsilon) \nabla_{\g\b}
u_\varepsilon^j+W_{u^jp_{\a}^i}(\nabla_\gamma u_\varepsilon,\nabla u_\varepsilon)\nabla_{\b}
u_\varepsilon^j\nonumber\\
&+W_{u^jp_{\a}^i}(u_\varepsilon,\nabla_\gamma\nabla u_\varepsilon)\nabla_{\b}
u_\varepsilon^j+W_{p_l^jp_{\a}^i}(u_\varepsilon,\nabla \nabla_{\b}u_\varepsilon) \nabla_{\b \g l}
u_\varepsilon^j\\
&+W_{u^jp_{\a}^i}(u_\varepsilon,\nabla \nabla_{\b}u_\varepsilon) \nabla_{\g} u_\varepsilon^j.
 \end{align*}
Since $W_{p_\alpha^ip_\beta^j}(z, u_\varepsilon)\xi_\alpha^i\xi_\beta^j\geq a|\xi|^2$ and
\begin{align*}
&|W_{u^jp_\alpha^i}(u_\varepsilon,\nabla u_\varepsilon)\nabla_{\gamma\beta}u_\varepsilon^j+W_{u^jp_{\a}^i}(\nabla_\gamma u_\varepsilon,\nabla u_\varepsilon)\nabla_{\b}
u_\varepsilon^j\\
+&W_{u^jp_\alpha^i}(u_\varepsilon,\nabla_\gamma\nabla u_\varepsilon)\nabla_\beta u_\varepsilon^j
+W_{u^jp_\alpha^i}(u_\varepsilon,\nabla_\beta\nabla u_\varepsilon)\nabla_\gamma u_\varepsilon^j|\leq C|\nabla u_\varepsilon||\nabla^2u_\varepsilon|,
\end{align*}
it follows that
\begin{align*}
-\int_{\mathbb R^3}W_{p_\alpha^ip_\beta^j}(u_\varepsilon,\nabla_\beta\nabla u_\varepsilon)\nabla_{\beta\gamma l}u_\varepsilon^j\nabla_{\gamma\beta\alpha}u_\varepsilon^idx
\leq-a\int_{\mathbb R^3}|\nabla^3u_\varepsilon|^2dx
\end{align*}
and
\begin{align*}
&\int_{\mathbb R^3}[W_{u^jp_\alpha^i}(u_\varepsilon,\nabla u_\varepsilon)\nabla_{\gamma\beta}u_\varepsilon^j+W_{u^jp_{\a}^i}(\nabla_\gamma u_\varepsilon,\nabla u_\varepsilon)\nabla_{\b}
u_\varepsilon^j\\
&+W_{u^jp_\alpha^i}(u_\varepsilon,\nabla_\gamma\nabla u_\varepsilon)\nabla_\beta u_\varepsilon^j
+W_{u^jp_\alpha^i}(u_\varepsilon,\nabla_\beta\nabla u_\varepsilon)\nabla_\gamma u_\varepsilon^j]\nabla_{\alpha\beta\gamma}u_\varepsilon^idx\nonumber\\
\leq&\eta\int_{\mathbb R^3}|\nabla^3u_\varepsilon|^2dx+C\int_{\mathbb R^3}(|\nabla u_\varepsilon|^2|\nabla^2u_\varepsilon|^2+|\nabla u_\varepsilon|^6)dx.
\end{align*}
Combining the above two inequalities yields
\begin{align}\label{2.5}
&-\int_{\R^3}\nabla_{\alpha\beta}(W_{p_{\a}^i} (u_\varepsilon,\nabla u_\varepsilon))
\nabla_{\b} \laplacian u_\varepsilon^i\,dx=-\int_{\R^3}\nabla^2_{\g \b}
W_{p_{\a}^i} (u_\varepsilon,\nabla u_\varepsilon)
\nabla^3_{\g \b \a} u_\varepsilon^i\,dx\nonumber\\
& \leq -(a-\eta)\int_{\R^3} |\nabla^3 u_\varepsilon|^2\,dx
+C\int_{\R^3}(|\nabla u_\varepsilon|^2|\nabla^2 u_\varepsilon|^2+|\nabla u_\varepsilon|^6)dx
\end{align}
for a sufficient small $\eta>0$.
Hence
\begin{align}
&\int_{\mathbb R^3}\nabla_\beta [W_{u^i}(u_\varepsilon,\nabla u_\varepsilon)]\nabla_\beta\Delta u_\varepsilon^idx\nonumber\\
=&\int_{\mathbb R^3}[W_{u^iu^j}(u_\varepsilon,\nabla u_\varepsilon)\nabla_\beta u_\varepsilon^j+W_{u^ip_\alpha^j}(u_\varepsilon,\nabla u_\varepsilon)\nabla_{\alpha\beta}u_\varepsilon^j]\nabla_\beta\Delta u_\varepsilon^idx\nonumber\\
\leq&\eta\int_{\mathbb R^3}|\nabla^3u_\varepsilon|^2dx+C\int_{\mathbb R^3}(|\nabla u_\varepsilon|^2|\nabla^2u_\varepsilon|^2+|\nabla u_\varepsilon|^6)dx.\label{2.6}
\end{align}

Now we estimate the second term on the right hand side of (\ref{L.2-1}).
Direct calculations give
\begin{align*}
u_\varepsilon^i\Delta^2u_\varepsilon^i=&\Delta(u_\varepsilon^i\Delta u_\varepsilon^i)-2\nabla u_\varepsilon^i\Delta\nabla u_\varepsilon^i-|\Delta u_\varepsilon^i|^2\\
=&\Delta\left(\frac{1}{2}\Delta|u_\varepsilon^i|^2-|\nabla u_\varepsilon^i|^2\right)-2\nabla u_\varepsilon^i\Delta\nabla u_\varepsilon^i-|\Delta u_\varepsilon^i|^2\\
=&\frac{1}{2}\Delta^2|u_\varepsilon^i|^2-2\nabla u_\varepsilon^i\nabla\Delta u_\varepsilon^i-2|\nabla^2u_\varepsilon^i|^2-2\nabla u_\varepsilon^i\Delta\nabla u_\varepsilon^i-|\Delta u_\varepsilon^i|^2\\
=&\frac{1}{2}\Delta^2|u_\varepsilon^i|^2-4\nabla u_\varepsilon^i\nabla\Delta u_\varepsilon^i-2|\nabla^2u_\varepsilon^i|^2-|\Delta u_\varepsilon^i|^2.
\end{align*}
Due to (\ref{L.2}), one has
\begin{align}
&-\int_{\mathbb R^3}\nabla_\beta\left[\frac{1}{\varepsilon^2}(1-|u_\varepsilon|^2)u_\varepsilon^i\right]\nabla_\beta\Delta u_\varepsilon^i dx=\int_{\mathbb R^3}\frac{|u_\varepsilon|^2 -1}{\varepsilon^2}u_\varepsilon^i\Delta^2 u_\varepsilon^idx\nonumber\\
=&\int_{\mathbb R^3}\frac{|u_\varepsilon|^2 -1}{\varepsilon^2}\left(\frac{1}{2}\Delta^2|u_\varepsilon^i|^2-4\nabla u_\varepsilon^i\nabla\Delta u_\varepsilon^i-2|\nabla^2u_\varepsilon^i|^2-|\Delta u_\varepsilon^i|^2\right)dx\nonumber\\
=&-\frac{1}{2\varepsilon^2}\int_{\mathbb R^3}|\Delta|u_\varepsilon|^2|^2dx+\int_{\mathbb R^3}\frac{|u_\varepsilon|^2 -1}{\varepsilon^2}(4\nabla u_\varepsilon^i\nabla\Delta u_\varepsilon^i\nonumber\\
&+2|\nabla^2u_\varepsilon^i|^2+|\Delta u_\varepsilon^i|^2)dx\nonumber\\
\leq&\eta\int_{\mathbb R^3}|\nabla\Delta u_\varepsilon|^2dx-\frac{1}{2\varepsilon^2}\int_{\mathbb R^3}|\Delta|u_\varepsilon|^2|^2dx+C\int_{\mathbb R^3}[|\nabla u_\varepsilon|^2(|\partial_tu+v_\varepsilon\nabla u_\varepsilon|^2\nonumber\\
&+|\nabla u_\varepsilon|^4+|\nabla^2u_\varepsilon|^2)+(|\partial_tu_\varepsilon+v_\varepsilon\nabla u_\varepsilon|+|\nabla^2u_\varepsilon|)|\nabla^2u_\varepsilon|^2]dx.\label{2.7}
\end{align}

Substituting (\ref{2.5})--(\ref{2.7}) into (\ref{L.2-1}) leads to
\begin{align}
&\frac{d}{dt}\int_{\mathbb R^3}|\Delta u_\varepsilon|^2dx+\int_{\mathbb R^3}\left(\frac{3a}{2}|\nabla^3u_\varepsilon|^2+\frac{1}{\varepsilon^2}|\Delta|u_\varepsilon|^2|^2\right)dx\nonumber\\
\leq&2\int_{\mathbb R^3}\nabla_\beta[(v_\varepsilon\cdot\nabla) u_\varepsilon]\nabla_\beta\Delta u_\varepsilon dx +C\int_{\mathbb R^3}[(|\partial_tu_\varepsilon+v_\varepsilon\nabla u_\varepsilon|+|\nabla^2u_\varepsilon|)|\nabla^2u_\varepsilon|^2\nonumber\\
&+|\nabla u_\varepsilon|^2(|\partial_tu_\varepsilon+v_\varepsilon\nabla u_\varepsilon|^2+|\nabla u_\varepsilon|^4+|\nabla^2u_\varepsilon|^2)]dx.\label{5.1-0}
\end{align}
Then it follows from Young inequality that
\begin{align}
&\frac{d}{dt}\int_{\mathbb R^3}|\Delta u_\varepsilon|^2dx+\int_{\mathbb R^3}\left(a|\nabla^3u_\varepsilon|^2+\frac{1}{\varepsilon^2}|\Delta|u_\varepsilon|^2|^2\right)dx\nonumber\\
\leq&C\int_{\mathbb R^3}[(|\nabla u_\varepsilon|^2+|v_\varepsilon|^2)(|\nabla^2u_\varepsilon|^2+|\partial_tu_\varepsilon|^2+|\nabla v_\varepsilon|^2)\nonumber\\
&+|\nabla^2u_\varepsilon|^2(|\nabla^2u_\varepsilon|+|\partial_tu_\varepsilon|)+|v_\varepsilon|^2|\nabla u_\varepsilon|^4+|\nabla u_\varepsilon|^6]dx.\label{5.1}
\end{align}

Differentiating equation (\ref{1.12}) with respect to $t$, multiplying the resulting equation by $\partial_tu_\varepsilon^i$ and integrating over $\mathbb R^3$, recalling (\ref{L.2}), we have
\begin{align*}
&\frac{d}{dt}\int_{\mathbb R^3}\frac{|\partial_tu_\varepsilon^i|^2}{2}dx+\int_{\mathbb R^3}W_{p_\alpha^ip_\beta^j}(u_\varepsilon,\nabla u_\varepsilon)\partial_t\nabla_\beta u_\varepsilon^j\partial_t\nabla_\alpha u_\varepsilon^idx\\
=&-\int_{\mathbb R^3}[W_{u^jp_\alpha^i}(u_\varepsilon,\nabla u_\varepsilon)\partial_tu_\varepsilon^j\nabla\partial_tu_\varepsilon^i+(v_\varepsilon\nabla\partial_tu_\varepsilon^i+\partial_tv_\varepsilon\nabla u_\varepsilon^i)\partial_tu_\varepsilon^i\\
&+W_{u^iu^j}(u_\varepsilon,\nabla u_\varepsilon)\partial_tu_\varepsilon^i\partial_tu_\varepsilon^j+W_{u^ip_\beta^j}(u_\varepsilon,\nabla u_\varepsilon)\nabla\partial_tu_\varepsilon^j\partial_tu_\varepsilon^i\\
&+\frac{1}{2\varepsilon^2}|\partial_t|u_\varepsilon|^2|^2+\frac{1}{\varepsilon^2}(|u_\varepsilon|^2-1)|\partial_tu_\varepsilon^i|^2]dx.
\end{align*}
This, together with $W_{p_\alpha^ip_\beta^j}(z,p)|\xi|^2\geq a|\xi|^2$ and $\int_{\mathbb R^3}v_\varepsilon\nabla\partial_tv_\varepsilon^i\partial_tv_\varepsilon^idx=0$, shows that
\begin{align}
&\frac{d}{dt}\int_{\mathbb R^3}|\partial_tu_\varepsilon|^2dx+\int_{\mathbb R^3}\left[\frac{3a}{2}|\nabla\partial_tu_\varepsilon|^2+\frac{|\partial_t|u_\varepsilon|^2|^2}{\varepsilon^2}+\frac{2(|u_\varepsilon|^2-1)}{\varepsilon^2}|\partial_tu_\varepsilon|^2\right]dx\nonumber\\
\leq&\eta\int_{\mathbb R^3}|\partial_tv_\varepsilon|^2dx+C\int_{\mathbb R^3}|\nabla u_\varepsilon|^2|\partial_tu_\varepsilon|^2dx.\label{5.2}
\end{align}

Due to the identity
$$
\partial_tu_\varepsilon=|u_\varepsilon|^{-2}(\partial_tu_\varepsilon\cdot u_\varepsilon)u_\varepsilon-|u_\varepsilon|^{-2}(\partial_tu_\varepsilon\times u_\varepsilon)\times u_\varepsilon,
$$
it holds
\begin{align}
&\int_{\mathbb R^3}\left[\frac{|\partial_t|u_\varepsilon|^2|^2}{\varepsilon^2}+\frac{2(|u_\varepsilon|^2-1)}{\varepsilon^2}|\partial_tu_\varepsilon|^2\right]dx\nonumber\\
=&\int_{\mathbb R^3}\left[\frac{|\partial_t|u_\varepsilon|^2|^2}{\varepsilon^2}+\frac{(|u_\varepsilon|^2-1)}{2\varepsilon^2|u_\varepsilon|^2}|\partial_t|u_\varepsilon|^2|^2\right.\nonumber\\
&-\left.\frac{2(|u_\varepsilon|^2-1)}{\varepsilon^2|u_\varepsilon|^2}((\partial_tu_\varepsilon\times u_\varepsilon)\times u_\varepsilon)\cdot\partial_tu_\varepsilon\right]dx\nonumber\\
\geq&\int_{\mathbb R^3}\left[\frac{|\partial_t|u_\varepsilon|^2|^2}{2\varepsilon^2}
-\frac{2(|u_\varepsilon|^2-1)}{\varepsilon^2|u_\varepsilon|^2}((\partial_tu_\varepsilon\times u_\varepsilon)\times u_\varepsilon)\cdot\partial_tu_\varepsilon\right]dx,\label{5.3}
\end{align}
where in the last step, the assumption $|u_\varepsilon|\geq\frac{3}{4}\geq\frac{1}{\sqrt 2}$ has been used.

Now, we estimate the term $\int_{\mathbb R^3}\frac{2(|u_\varepsilon|^2-1)}{\varepsilon^2|u_\varepsilon|^2}((\partial_tu_\varepsilon\times u_\varepsilon)\times u_\varepsilon)\cdot\partial_tu_\varepsilon dx$ in (\ref{5.3}).
It follows (\ref{1.12}) that
\begin{align*}
\frac{(|u_\varepsilon|^2-1)}{\varepsilon^2|u_\varepsilon|^2}=&\frac{-2}{|u_\varepsilon|^4}\left(\frac{1}{2}\partial_t|u_\varepsilon|^2
-\nabla_\alpha(W_{p_\alpha}(u_\varepsilon,\nabla u_\varepsilon))\cdot u_\varepsilon\right.\\
&+(W_u(u_\varepsilon,\nabla u_\varepsilon)+(v_\varepsilon\cdot \nabla )  u_\varepsilon)\cdot u_\varepsilon\Big),
\end{align*}
and thus
\begin{align*}
&\int_{\mathbb R^3}\frac{2(|u_\varepsilon|^2-1)}{\varepsilon^2|u_\varepsilon|^2}{((\partial_tu_\varepsilon\times u_\varepsilon)\times u_\varepsilon)\cdot\partial_tu_\varepsilon}dx\\
=&\int_{\mathbb R^3}\frac{-2}{|u_\varepsilon|^4}\left(\frac{1}{2}\partial_t|u_\varepsilon|^2-\nabla_\alpha(W_{p_\alpha}(u_\varepsilon,\nabla u_\varepsilon))\cdot u_\varepsilon\right.\\
&+(W_u(u_\varepsilon,\nabla u_\varepsilon)+v_\varepsilon\nabla u_\varepsilon)\cdot u_\varepsilon\Large)[((\partial_tu_\varepsilon\times u_\varepsilon)\times u_\varepsilon)\cdot\partial_tu_\varepsilon]dx\\
=&\int_{\mathbb R^3}\frac{-2}{|u_\varepsilon|^4}(W_u(u_\varepsilon,\nabla u_\varepsilon)+v_\varepsilon\nabla u_\varepsilon)\cdot u_\varepsilon((\partial_tu_\varepsilon\times u_\varepsilon)\times u_\varepsilon)\cdot\partial_tu_\varepsilon dx\\
&+\int_{\mathbb R^3}\frac{2}{|u_\varepsilon|^4}\nabla_\alpha(W_{p_\alpha}(u_\varepsilon,\nabla u_\varepsilon))\cdot u_\varepsilon((\partial_tu_\varepsilon\times u_\varepsilon)\times u_\varepsilon)\cdot\partial_tu_\varepsilon dx\\
&+\int_{\mathbb R^3}\frac{-1}{|u_\varepsilon|^4}\partial_t|u_\varepsilon|^2((\partial_tu_\varepsilon\times u_\varepsilon)\times u_\varepsilon)\cdot\partial_tu_\varepsilon dx\\
=&I_1+I_2+I_3.
\end{align*}

To estimate $I_1$ and $I_2$, we have
$$
I_1\leq C\int_{\mathbb R^3}(|v_\varepsilon|^2+|\nabla u_\varepsilon|^2)|\partial_tu_\varepsilon|^2dx,
$$
and
\begin{align*}
I_2=&-\int_{\mathbb R^3}W_{p_\alpha}(u_\varepsilon,\nabla u_\varepsilon)\nabla_\alpha\left(\frac{2}{|u_\varepsilon|^4}u_\varepsilon((\partial_tu_\varepsilon\times u_\varepsilon)\times u_\varepsilon)\cdot\partial_tu_\varepsilon\right)dx\\
\leq&C\int_{\mathbb R^3}|\nabla u_\varepsilon|(|\partial_tu_\varepsilon||\nabla\partial_tu_\varepsilon|+|\nabla u_\varepsilon||\partial_tu_\varepsilon|^2)dx\\
\leq&\eta\int_{\mathbb R^3}|\nabla\partial_tu_\varepsilon|^2dx+C\int_{\mathbb R^3}|\nabla u_\varepsilon|^2|\partial_tu_\varepsilon|^2dx.
\end{align*}
To estimate $I_3$, we use (\ref{1.12}) to obtain
$$
\partial_tu_\varepsilon\times u_\varepsilon=[\nabla_\alpha( W_{p_\alpha}(u_\varepsilon,\nabla u_\varepsilon))-W_u(u_\varepsilon,\nabla u_\varepsilon)-(v_\varepsilon\cdot\nabla u_\varepsilon)]\times u_\varepsilon.
$$
Thus
\begin{align*}
I_3=&\int_{\mathbb R^3}\frac{-1}{|u_\varepsilon|^4}\partial_t|u_\varepsilon|^2[(\nabla_\alpha(W_{p_\alpha}(u_\varepsilon,\nabla u_\varepsilon))-W_u(u_\varepsilon,\nabla u_\varepsilon)\\
&-(v_\varepsilon\cdot\nabla) u_\varepsilon)\times u_\varepsilon\times u_\varepsilon]\cdot\partial_tu_\varepsilon dx\\
\leq&C\int_{\mathbb R^3}[|W_{p_\alpha}(u_\varepsilon,\nabla u_\varepsilon)|(|\nabla u_\varepsilon||\partial_tu_\varepsilon|^2+|\partial_t|u_\varepsilon|^2||\nabla\partial_tu_\varepsilon|\\
&+|\nabla\partial_t|u_\varepsilon|^2||\partial_tu_\varepsilon|)+|\partial_t|u_\varepsilon|^2||\partial_tu_\varepsilon|(|\nabla u_\varepsilon|^2+|v_\varepsilon|^2)]dx\\
\leq&C\int_{\mathbb R^3}[|\nabla u_\varepsilon|(|\nabla u_\varepsilon||\partial_tu_\varepsilon|^2+|\partial_tu_\varepsilon||\nabla\partial_tu_\varepsilon|)+|\partial_tu_\varepsilon|^2(|\nabla u_\varepsilon|^2+|v_\varepsilon|^2)]dx\\
\leq&\eta\int_{\mathbb R^3}|\nabla\partial_tu_\varepsilon|^2dx+C\int_{\mathbb R^3}(|\nabla u_\varepsilon^2+|v_\varepsilon|^2)|\partial_tu_\varepsilon|^2dx.
\end{align*}
Combining above estimates of $I_1, I_2, I_3$ shows
\begin{align*}
&\int_{\mathbb R^3}\frac{2(|u_\varepsilon|^2-1)}{\varepsilon^2|u_\varepsilon|^2}{((\partial_tu_\varepsilon\times u_\varepsilon)\times u_\varepsilon)\cdot\partial_tu_\varepsilon}dx\nonumber\\
\leq&\eta\int_{\mathbb R^3}|\nabla\partial_tu_\varepsilon|^2dx+C\int_{\mathbb R^3}(|\nabla u_\varepsilon^2+|v_\varepsilon|^2)|\partial_tu_\varepsilon|^2dx,
\end{align*}
which, together with (\ref{5.2})--(\ref{5.3}), shows
\begin{align}
&\frac{d}{dt}\int_{\mathbb R^3}|\partial_tu_\varepsilon|^2dx+\int_{\mathbb R^3}\left(a|\nabla\partial_tu_\varepsilon|^2+\frac{1}{\varepsilon^2}|\partial_t|u_\varepsilon|^2|^2\right)dx\nonumber\\
\leq&\eta\int_{\mathbb R^3}|\partial_tv_\varepsilon|^2dx+C\int_{\mathbb R^3}(|\nabla u_\varepsilon|^2+|v_\varepsilon|^2)|\partial_tu_\varepsilon|^2dx. \label{5.4}
\end{align}

Multiplying equation (\ref{1.10}) by $\partial_t v^i_\varepsilon-\Delta v_\varepsilon^i$ and integrating over $\mathbb R^3$ yields
\begin{align*}
&\frac{d}{dt}\int_{\mathbb R^3}|\nabla v_\varepsilon|^2dx+\int_{\mathbb R^3}(|\partial_tv_\varepsilon|^2+|\Delta v_\varepsilon|^2)dx\nonumber\\
=&-\int_{\mathbb R^3}(v_\varepsilon\cdot\nabla v_\varepsilon^i+\nabla_j(\nabla_iu_\varepsilon^kW_{p_j^k}(u_\varepsilon,\nabla u_\varepsilon)))(\partial_t{v_\varepsilon^i}-\Delta {v_\varepsilon^i})dx\nonumber\\
\leq&\eta\int_{\mathbb R^3}(|\partial_t{v_\varepsilon^i}|^2+|\Delta {v_\varepsilon^i}|^2)dx+C\int_{\mathbb R^3}(|v_\varepsilon|^2|\nabla v_\varepsilon|^2\\
&+|\nabla u_\varepsilon|^2|\nabla^2u_\varepsilon|^2+|\nabla u_\varepsilon|^6)dx
\end{align*}
for a sufficient small $\eta>0$. Therefore
\begin{align}
&\frac{d}{dt}\int_{\mathbb R^3}|\nabla v_\varepsilon|^2dx+\frac{3}{4}\int_{\mathbb R^3}(|\partial_tv_\varepsilon|^2+|\Delta v_\varepsilon|^2)dx\nonumber\\
\leq&C\int_{\mathbb R^3}(|v_\varepsilon|^2|\nabla v_\varepsilon|^2+|\nabla u_\varepsilon|^2|\nabla^2u_\varepsilon|^2+|\nabla u_\varepsilon|^6)dx.\label{5.5}
\end{align}

It follows from (\ref{5.1}), (\ref{5.4}) and (\ref{5.5}) that
\begin{align}
&\frac{d}{dt}\int_{\mathbb R^3}(|\Delta u_\varepsilon|^2+|\partial_t u_\varepsilon|^2+|\nabla v_\varepsilon|^2)dx+\int_{\mathbb R^3}\left(a|\nabla^3u_\varepsilon|^2\right.\nonumber\\
&\left.+a|\nabla\partial_tu_\varepsilon|^2+|\Delta v_\varepsilon|^2+|\partial_tv_\varepsilon|^2+\frac{1}{\varepsilon^2}|\Delta|u_\varepsilon|^2|^2+\frac{1}{\varepsilon^2}|\partial_t|u_\varepsilon|^2|^2\right)dx\nonumber\\
\leq&C\int_{\mathbb R^3}[(|\nabla u_\varepsilon|^2+|v_\varepsilon|^2)(|\nabla^2u_\varepsilon|^2+|\partial_tu_\varepsilon|^2+|\nabla v_\varepsilon|^2)\nonumber\\
&+|\nabla^2u_\varepsilon|^2(|\nabla^2u_\varepsilon|+|\partial_tu_\varepsilon|)+|v_\varepsilon|^2|\nabla u_\varepsilon|^4+|\nabla u_\varepsilon|^6]dx.\label{5.6}
\end{align}
Note that
\begin{align}
&\int_{\mathbb R^3}|\nabla^2u_\varepsilon|^3dx=\int_{\mathbb R^3}|{\nabla_{ij}u_\varepsilon}|{\nabla_{ij}u_\varepsilon}\cdot{\nabla_{ij}u_\varepsilon}dx\nonumber\\
=&-\int_{\mathbb R^3}\nabla_i(|{\nabla_{ij}u_\varepsilon}|{\nabla_{ij}u_\varepsilon})\nabla_ju_\varepsilon dx\leq C\int_{\mathbb R^3}|\nabla^2u_\varepsilon||\nabla^3u_\varepsilon||\nabla u_\varepsilon|dx\nonumber\\
\leq&\eta\int_{\mathbb R^3}|\nabla^3u_\varepsilon|^2dx+C\int_{\mathbb R^3}|\nabla u_\varepsilon|^2|\nabla^2u_\varepsilon|^2dx\label{5.8}
\end{align}
and
\begin{align}
&\int_{\mathbb R^3}|\nabla^2u_\varepsilon|^2|\partial_tu_\varepsilon|dx\nonumber\\
=&\int_{\mathbb R^3}{\nabla_{ij}u_\varepsilon}\cdot{\nabla_{ij}u_\varepsilon}|\partial_tu_\varepsilon|dx=-\int_{\mathbb R^3}\nabla_i({\nabla_{ij}u_\varepsilon}|\partial_tu_\varepsilon|)\cdot\nabla_ju_\varepsilon dx\nonumber\\
\leq& C\int_{\mathbb R^3}|\nabla u_\varepsilon|(|\nabla^3u_\varepsilon||\partial_tu_\varepsilon|+|\nabla^2u_\varepsilon||\nabla\partial_tu_\varepsilon|)dx\nonumber\\
\leq&\eta\int_{\mathbb R^3}(|\nabla^3u_\varepsilon|^2+|\nabla\partial_tu_\varepsilon|^2)dx+C\int_{\mathbb R^3}|\nabla u_\varepsilon|^2(|\partial_tu_\varepsilon|^2+|\nabla^2u_\varepsilon|^2)dx\label{5.9}
\end{align}
for sufficient small $\eta>0$.
On the other hand, integrating by parts gives
\begin{align*}
&\int_{\mathbb R^3}|v_\varepsilon|^2|\nabla u_\varepsilon|^4dx=-\int_{\mathbb R^3}\textmd{div}(|v_\varepsilon|^2|\nabla u_\varepsilon|^2\nabla u_\varepsilon)(u_\varepsilon -b)dx\\
\leq&C\int_{\mathbb R^3}(|v_\varepsilon||\nabla v_\varepsilon||\nabla u_\varepsilon|^3+|v_\varepsilon|^2|\nabla u_\varepsilon|^2|\nabla^2u_\varepsilon|)dx\\
\leq&\frac{1}{2}\int_{\mathbb R^3}|v_\varepsilon|^2|\nabla u_\varepsilon|^4dx+C\int_{\mathbb R^3}(|\nabla u_\varepsilon|^2|\nabla v_\varepsilon|^2+|v_\varepsilon|^2|\nabla^2u_\varepsilon|^2)dx
\end{align*}
and
\begin{align*}
&\int_{\mathbb R^3}|\nabla u_\varepsilon|^6dx=-\int_{\mathbb R^3}\textmd{div}(|\nabla u_\varepsilon|^4\nabla u_\varepsilon)(u_\varepsilon -b)dx\\
\leq&C\int_{\mathbb R^3}|\nabla u_\varepsilon|^4|\nabla^2u_\varepsilon|dx\leq\frac{1}{2}\int_{\mathbb R^3}|\nabla u_\varepsilon|^6dx+C\int_{\mathbb R^3}|\nabla u_\varepsilon|^2|\nabla^2u_\varepsilon|^2dx.
\end{align*}
These imply
\begin{align}
\int_{\mathbb R^3}|v_\varepsilon|^2|\nabla u_\varepsilon|^4dx
\leq C\int_{\mathbb R^3}(|\nabla u_\varepsilon|^2|\nabla v_\varepsilon|^2+|v_\varepsilon|^2|\nabla^2u_\varepsilon|^2)dx\label{5.10}
\end{align}
and
\begin{align}
\int_{\mathbb R^3}|\nabla u_\varepsilon|^6dx \leq C\int_{\mathbb R^3}|\nabla u_\varepsilon|^2|\nabla^2u_\varepsilon|^2dx.\label{5.11}
\end{align}

Substituting (\ref{5.8})--(\ref{5.11}) into (\ref{5.6}) leads to
\begin{align*}
&\frac{d}{dt}\int_{\mathbb R^3}(|\Delta u_\varepsilon|^2+|\partial_t u_\varepsilon|^2+|\nabla v_\varepsilon|^2)dx+\int_{\mathbb R^3}\left(a|\nabla^3u_\varepsilon|^2+a|\nabla\partial_tu_\varepsilon|^2\right.\nonumber\\
&\left.+|\Delta v_\varepsilon|^2+|\partial_tv_\varepsilon|^2+\frac{1}{\varepsilon^2}|\Delta|u_\varepsilon|^2|^2+\frac{1}{\varepsilon^2}|\partial_t|u_\varepsilon|^2|^2\right)dx\nonumber\\
\leq&C\int_{\mathbb R^3}(|\nabla u_\varepsilon|^2+|v_\varepsilon|^2)(|\nabla^2u_\varepsilon|^2+|\partial_tu_\varepsilon|^2+|\nabla v_\varepsilon|^2)dx,
\end{align*}
which completes the proof.
\end{proof}

Due to the above lemma, we can prove the following uniform estimates (independent of $\varepsilon$) on the strong solutions to the system (\ref{1.10})--(\ref{1.12}).

\begin{prop}
\label{prop}
Suppose that the initial data $(u_{0\varepsilon}, v_{0\varepsilon})$ satisfies
\begin{align*}
&\frac{7}{8}\leq|u_{0\varepsilon}|\leq\frac{9}{8},\quad u_{0\varepsilon}-b\in H^2(\mathbb R^3),\quad v_{0\varepsilon}\in H^1(\mathbb R^3),\quad  \mbox{ div } v_{0\varepsilon} =0 \,\mbox { in }\R^3\\
&\|(\nabla u_{0\varepsilon}, v_{0\varepsilon}\|_{H^1(\mathbb R^3)}^2+\left\|\mathcal Q_\varepsilon(u_{0\varepsilon}, v_{0\varepsilon})\right\|_{L^2(\mathbb R^3)}^2\leq M^2
\end{align*}
for some positive constant $M$ and constant unit vector $b$, where
$$
\mathcal Q_{\varepsilon}(u, v)=\nabla_\alpha(W_{p_\alpha}(u,\nabla u))-W_u(u,\nabla u)+\frac{1-|u|^2}{\varepsilon^2}u-(v\cdot\nabla)u.
$$
Then there is an absolute constant $C^*>0$  such that the system (\ref{1.10})--(\ref{1.12}) with initial data $(u_{0\varepsilon}, v_{0\varepsilon})$ has a unique strong solutions $(u_\varepsilon, v_\varepsilon)$ in $\mathbb R^3\times(0, T_M)$ with $T_M=C^*M^{-4}$, satisfying
$$
\frac{7}{8}\leq|u_{\varepsilon}|\leq\frac{9}{8} \quad\mbox{ on }\mathbb R^3\times[0, C^*M^{-4}]
$$
and
\begin{align*}
&\sup_{0\leq t\leq T_M}\int_{\mathbb R^3}(|\Delta u_\varepsilon|^2+|\partial_tu_\varepsilon|^2+|\nabla v_\varepsilon|^2)dx\\
&+\int_0^{T_M}\int_{\mathbb R^3}(|\nabla\Delta u_\varepsilon |^2+|\nabla\partial_tu_\varepsilon|^2+|\Delta v_\varepsilon|^2+|\partial_tv_\varepsilon|^2)dxdt\leq C^*M^{-4},
\end{align*}
provided $\varepsilon\leq\varepsilon_M$, where $\varepsilon_M$ is a positive constant depending only on $M$.
\end{prop}

\begin{proof}
By Lemma \ref{LLEM2.1}, there is a unique local solution to the system (\ref{1.10})--(\ref{1.12}) with initial data $(u_{0\varepsilon}, v_{0\varepsilon})$, which can be extended to the maximum time $T_\varepsilon^{\mbox{max}}$. Note that the properties of $u$ stated in Lemma \ref{LLEM2.1} impliy that $u$ is H\"older continuous on $\mathbb R^3\times[0, T_\varepsilon^{\mbox{max}})$ due to the well-known Gagliado-Nirenberg-Sobolev inequality. Since $\frac{7}{8}\leq|u_{0\varepsilon}|\leq\frac{9}{8}$, there is a maximal time $T_\varepsilon^1\in(0, T_\varepsilon^{\mbox{max}}]$, such that $\frac{3}{4}\leq|u_\varepsilon|\leq\frac{5}{4}$ on $\mathbb R^3\times[0, T_\varepsilon^1)$.

It follows from Lemma \ref{lem5.1} that
\begin{align*}
&\frac{d}{dt}\int_{\mathbb R^3}(|\Delta u_\varepsilon|^2+|\partial_t u_\varepsilon|^2+|\nabla v_\varepsilon|^2)dx\\
&+\int_{\mathbb R^3}\left(a|\nabla^3u_\varepsilon|^2+a|\nabla\partial_tu_\varepsilon|^2+|\Delta v_\varepsilon|^2+|\partial_tv_\varepsilon|^2\right)dx\nonumber\\
\leq&C\int_{\mathbb R^3}(|\nabla u_\varepsilon|^2+|v_\varepsilon|^2)(|\nabla^2u_\varepsilon|^2+|\partial_tu_\varepsilon|^2+|\nabla v_\varepsilon|^2)dx
\end{align*}
for any $t\in(0, T_\varepsilon^1)$. Using the Gagliado-Nirenberg-Sobolev inequality in the above inequality yields
\begin{align*}
&\frac{d}{dt}\int_{\mathbb R^3}(|\Delta u_\varepsilon|^2+|\partial_t u_\varepsilon|^2+|\nabla v_\varepsilon|^2)dx\\
&+\int_{\mathbb R^3}\left(a|\nabla^3u_\varepsilon|^2+a|\nabla\partial_tu_\varepsilon|^2+|\Delta v_\varepsilon|^2+|\partial_tv_\varepsilon|^2\right)dx\nonumber\\
\leq&C\left[\int_{\mathbb R^3}(|\nabla u_\varepsilon|^6+|v_\varepsilon|^6)dx\right]^{1/3}\left[\int_{\mathbb R^3}(|\nabla^2u_\varepsilon|^2+|\partial_tu_\varepsilon|^2+|\nabla v_\varepsilon|^2)dx\right]^{1/2}\\
&\times\left[\int_{\mathbb R^3}(|\nabla^2u_\varepsilon|^6+|\partial_tu_\varepsilon|^6+|\nabla v_\varepsilon|^6)dx\right]^{1/6}\\
\leq&C\left[\int_{\mathbb R^3}(|\Delta u_\varepsilon|^2+|\nabla v_\varepsilon|^2)dx\right]\left[\int_{\mathbb R^3}(|\Delta u_\varepsilon|^2+|\partial_tu_\varepsilon|^2+|\nabla v_\varepsilon|^2)dx\right]^{1/2}\\
&\times\left[\int_{\mathbb R^3}(|\nabla\Delta u_\varepsilon |^2+|\nabla\partial_tu_\varepsilon|^2+|\Delta v_\varepsilon|^2)dx\right]^{1/2}\\
\leq&\eta\int_{\mathbb R^3}(|\nabla\Delta u_\varepsilon|^2+|\nabla\partial_tu_\varepsilon|^2+|\Delta v_\varepsilon|^2)dx\\
&+C\left[\int_{\mathbb R^3}(|\Delta u_\varepsilon|^2+|\partial_tu_\varepsilon|^2+|\nabla v_\varepsilon|^2)dx\right]^{3}
\end{align*}
and thus
\begin{align}
&\frac{d}{dt}\int_{\mathbb R^3}(|\Delta u_\varepsilon|^2+|\partial_t u_\varepsilon|^2+|\nabla v_\varepsilon|^2)dx\nonumber\\
&+\int_{\mathbb R^3}\left(\frac{a}{2}|\nabla^3u_\varepsilon|^2+\frac{a}{2}|\nabla\partial_tu_\varepsilon|^2+|\Delta v_\varepsilon|^2+|\partial_tv_\varepsilon|^2\right)dx\nonumber\\
\leq&C\left[\int_{\mathbb R^3}(|\Delta u_\varepsilon|^2+|\partial_tu_\varepsilon|^2+|\nabla v_\varepsilon|^2)dx\right]^{3}\label{L.8}
\end{align}
for any $t\in(0, T_\varepsilon^1)$.

Define
\begin{align*}
f(t)=&\int_{\mathbb R^3}\left(|\Delta u_\varepsilon|^2+|\partial_tu_\varepsilon|^2+|\nabla v_\varepsilon|^2\right)dx\nonumber\\
&+\int_0^t\int_{\mathbb R^3}\left(\frac{a}{2}|\nabla^3u_\varepsilon|^2+\frac{a}{2}|\nabla\partial_tu_\varepsilon|^2+|\Delta v_\varepsilon|^2+|\partial_tv_\varepsilon|^2\right)dxd\tau.
\end{align*}
It follows from (\ref{L.8}) that
$$
f'(t)\leq C_1f(t)^3,
$$
which implies
\begin{equation*}
f(t)\leq\left(\frac{1}{1-2C_1f(0)^2t}\right)^{1/2}f(0)\leq2f(0)
\end{equation*}
for any $t\leq \min\left\{T_\varepsilon^1, \frac{3}{8C_1f(0)^2}\right\}$.
Due to equation (\ref{1.12}), it holds that
$$
f(0)=\int_{\mathbb R^3}(|\Delta u_{0\varepsilon}|^2+|\nabla v_{0\varepsilon}|^2+|\mathcal Q_{\varepsilon}(u_{0\varepsilon}, v_{0\varepsilon})|^2)dx\leq C_2M^2
$$
and thus
\begin{equation}
f(t)\leq 2C_2M^2,\qquad\forall t\leq\min\{T_\varepsilon^1, C_3M^{-4}\}, \label{2.18}
\end{equation}
where $C_3=\frac{3}{8C_1C_2^2}$.

By Lemma \ref{LLEM2.2}, one has
$$
\sup_{0\leq t\leq T_\varepsilon^{\text{max}}}(\|(\nabla u_\varepsilon, v_\varepsilon)\|_{L^2}^2+\varepsilon^{-2}\|1-|u_\varepsilon|^2\|_{L^2}^2)\leq CM^2.
$$
Combining the above inequality with (\ref{2.18}) and using Gagliado-Nirenberg-Sobolev inequality, we deduce
\begin{align*}
\|1-|u_\varepsilon|^2\|_{L^\infty(\mathbb R^3)}\leq&C\|1-|u_\varepsilon|^2\|_{L^2(\mathbb R^3)}^{1/4}\|\nabla^2(1-|u_\varepsilon|^2)\|_{L^2(\mathbb R^3)}^{3/4}\\
\leq&C(\varepsilon M)^{1/4}(\|\nabla^2u_\varepsilon\|_{L^2(\mathbb R^3)}^{3/4}+\|\nabla u_\varepsilon\|_{L^4(\mathbb R^3)}^{3/4})\\
\leq&C(\varepsilon M)^{1/4}(1+\|\nabla u_\varepsilon\|_{H^1(\mathbb R^3)}^{3/2})\\
\leq&C(\varepsilon M)^{1/4}(1+M^{3/2})\leq C_4(1+M^{7/4})\varepsilon^{1/4}
\end{align*}
for any $t\leq \min\{T_\varepsilon^1, C_3M^{-4}\}$, and thus
$$
\frac{7}{8}\leq|u_\varepsilon|\leq\frac{9}{8},\qquad\text{on }\mathbb R^3\times[0, \min\{T_\varepsilon^1, C_3M^{-4}\}],
$$
provided $\varepsilon\leq\varepsilon_M:=\left[\frac{(9/8)^2-1}{C_4(1+M^{7/4})}\right]^4. $ Note that (\ref{2.18}) implies $\min\{T_\varepsilon^1, C_3M^{-4}\}<T_\varepsilon^{\text{max}}$, otherwise we can extend $(u_\varepsilon, v_\varepsilon)$ beyond $T_\varepsilon^{\text{max}}$, contradicting to the definition of $T_\varepsilon^{\text{max}}$. Due to this fact and the above inequality, there is another time $T_\varepsilon^2$ with
$$
\min\{T_\varepsilon^1, C_3M^{-4}\}<T_\varepsilon^2\leq T_\varepsilon^{\text{max}},
$$
such that
$$
\frac{3}{4}\leq|u_\varepsilon|\leq\frac{5}{4}\quad\mbox{on }\mathbb R^3\times[0, T_\varepsilon^2).
$$
The definition of $T_\varepsilon^1$ gives $T_\varepsilon^1\geq T_\varepsilon^2$, and thus
$$
\min\{T_\varepsilon^1, C_3M^{-4}\}<T_\varepsilon^2\leq T_\varepsilon^1,
$$
which forces $T_\varepsilon^1>C_3M^{-4}.$ As a result, it follows from (\ref{2.18}) that
$$
f(t)\leq 2C_2M^2,\qquad\forall t\leq C_3M^{-4},
$$
which implies the conclusion with $C^*=C_3$.
\end{proof}

We will use the following version of the Aubin-Lions lemma.

\begin{lemma}\label{lem2.3}
(Aubin-Lions Lemma, See Simon \cite{Simon} Corollary 4) Assume that $X, B$ and $Y$ are three Banach spaces, with $X\hookrightarrow\hookrightarrow B\hookrightarrow Y.$ Then it holds that

(i) If $F$ is a bounded subset of $L^p(0, T; X)$ where $1\leq p<\infty$, and $\frac{\partial F}{\partial t}=\left\{\frac{\partial f}{\partial t}|f\in F\right\}$ is bounded in $L^1(0, T; Y)$, then $F$ is relatively compact in $L^p(0, T; B)$;

(ii) If $F$ is bounded in $L^\infty(0, T; X)$ and $\frac{\partial F}{\partial t}$ is bounded in $L^r(0, T; Y)$ where $r>1$, then $F$ is relatively compact in $C([0, T]; B)$.
\end{lemma}

Now we can prove the local existence and uniqueness of strong solutions to the Ericksen-Leslie system.

\begin{proof} [Proof of the local existence and uniqueness part of Theorem \ref{thm1}]
  For any $\varepsilon>0$, by Proposition \ref{prop}, there is a positive number $T$ independent of $\varepsilon$, such that the system (\ref{1.10})--(\ref{1.12}) with the initial condition (\ref{1.9}) has a unique solution $(u_\varepsilon, v_\varepsilon)$, with the properties
\begin{eqnarray*}
&&u_\varepsilon\in L^2(0,T; H_b^3(\mathbb R^3)),\quad\partial_tu_\varepsilon\in L^2(0, T; H^1(\mathbb R^3)),\\
&&v_\varepsilon\in L^2(0, T; H^2_\sigma(\mathbb R^3)),\quad\partial_tv_\varepsilon\in L^2(\mathbb R^3\times(0, T)),\\
&&\frac{7}{8}\leq|u_\varepsilon|\leq\frac{9}{8}\mbox{ on }\mathbb R^3\times(0,T)
\end{eqnarray*}
and
\begin{align*}
&\sup_{0\leq t\leq T}\int_{\mathbb R^3}\left(\frac{(1-|u_\varepsilon|^2)^2}{\varepsilon^2}+|\nabla u_\varepsilon|^2+|\nabla^2u_\varepsilon|^2+|\partial_tu_\varepsilon|^2+|v_\varepsilon|^2+|\nabla v_\varepsilon|^2\right)dx\\
&+\int_0^T\int_{\mathbb R^3}(|\nabla^2u_\varepsilon|^2+|\partial_tu_\varepsilon|^2+|\nabla^3u_\varepsilon|^2+|\nabla\partial_tu_\varepsilon|^2+|\nabla v_\varepsilon|^2+|\nabla^2v_\varepsilon|^2\\
&+|\partial_tv_\varepsilon|^2)dxdt\leq C.
\end{align*}

Due to (\ref{1.10}) and (\ref{1.11}), the pressure $p_\varepsilon$ satisfies
$$
\Delta p_\varepsilon=-\nabla_{ij}^2(v_\varepsilon^i v_\varepsilon^j+\nabla_iu_\varepsilon^kW_{p_{j}^k}(u_\varepsilon,\nabla u_\varepsilon))
$$
and
\begin{align*}
\Delta p_\varepsilon=&-\nabla\cdot((v_\varepsilon\cdot\nabla)v_\varepsilon)-\nabla_i[\nabla_{ij}^2u_\varepsilon^kW_{p_j^k}(u_\varepsilon,\nabla u_\varepsilon)\\
&+\nabla_i u_\varepsilon^k(W_{p_j^kp_\alpha^l}(u_\varepsilon,\nabla u_\varepsilon)\nabla_{j\alpha}^2u_\varepsilon^l+W_{u^lp_j^k}(u_\varepsilon,\nabla u_\varepsilon)\nabla_ju_\varepsilon^l)]
\end{align*}
from which, using elliptic estimates, we obtain
$$
\int_0^T\int_{\mathbb R^3}|p_\varepsilon|^2dxdt\leq C\int_0^T\int_{\mathbb R^3}(|v_\varepsilon|^4+|\nabla u_\varepsilon|^4)dxdt\leq C
$$
and
\begin{align*}
&\int_0^T\int_{\mathbb R^3}|\nabla p_\varepsilon|^2dxdt\\
\leq& C\int_0^T\int_{\mathbb R^3}[|v_\varepsilon|^2|\nabla v_\varepsilon|^2+|\nabla^2u_\varepsilon|^2|\nabla u_\varepsilon|^2+|\nabla u_\varepsilon|^2(|\nabla^2u_\varepsilon|^2+|\nabla u_\varepsilon|^4)]dxdt\\
\leq&\int_0^T\int_{\mathbb R^3}(|v_\varepsilon|^6+|\nabla u_\varepsilon|^6+|\nabla v_\varepsilon|^3+|\nabla^2u_\varepsilon|^3)dxdt\leq C
\end{align*}
for some positive constant $C$ independent of $\varepsilon$. In the above, we have used the Gagliado-Nirenberg-Sobolev inequality and the estimates stated in the previous.

On account of all the estimates obtained in the above, there is a subsequence, still denoted by $(u_\varepsilon, v_\varepsilon, p_\varepsilon)$, and $(u, v, p)$, such that
\begin{equation}\label{2.17}
\begin{array}{l}
u\in L^2(0,T; H_b^3(\mathbb R^3; S^2),\quad \partial_tu\in L^2(0, T; H^1(\mathbb R^3)),\\
v\in L^2(0, T; H^2(\mathbb R^3)),\quad\partial_tv\in L^2(\mathbb R^3\times(0, T)),\\
p\in L^2(0, T; H^1(\mathbb R^3)),\quad(u, v)\mbox{ satisfies the initial condition }
\end{array}
\end{equation}
and for any $R\in(0, \infty)$
\begin{eqnarray*}
&&u_\varepsilon\rightarrow u\mbox{ in }L^2(0, T; H^2(B_R))\cap C([0, T]; H^1(B_R)),\\
&&u_\varepsilon\rightharpoonup u\mbox{ in }L^2(0, T; H^3(\mathbb R^3)),\quad \partial_tu_\varepsilon\rightharpoonup\partial_tu\mbox{ in }L^2(0, T; H^1(\mathbb R^3)),\\
&&v_\varepsilon\rightarrow v\mbox{ in }L^2(0, T; H^1(B_R))\cap L^2([0,T ]; L^2(B_R)),\\
&&v_\varepsilon\rightharpoonup v\mbox{ in }L^2(0, T; H^2(\mathbb R^3)),\quad\partial_tv_\varepsilon\rightharpoonup\partial_t u\mbox{ in }L^2(\mathbb R^3\times(0, T)),\\
&&p_\varepsilon\rightharpoonup p\mbox{ in }L^2(0, T; L^2(\mathbb R^3)),
\end{eqnarray*}
where $|u|=1$ follows from the estimate that $\sup_{0\leq t\leq T}\int_{\mathbb R^3}\frac{(1-|u_\varepsilon|^2)^2}{\varepsilon^2}dx\leq C$, while the strong convergence stated above follows from the Aubin-Lions lemma.

By (\ref{1.12}), we have
\begin{align}
&\partial_tu_\varepsilon\times u_\varepsilon+(v_\varepsilon\cdot\nabla)u_\varepsilon\times u_\varepsilon=\nabla_\alpha[W_{p_\alpha}(u_\varepsilon,\nabla u_\varepsilon)]\times u_\varepsilon-W_{u}(u_\varepsilon,\nabla u_\varepsilon)\times u_\varepsilon\nonumber\\
=&\nabla_\alpha[W_{p_\alpha}(u_\varepsilon,\nabla u_\varepsilon)\times u_\varepsilon]-W_{p_\alpha}(u_\varepsilon,\nabla u_\varepsilon)\times\nabla_\alpha u_\varepsilon-W_u(u_\varepsilon,\nabla u_\varepsilon)\times u_\varepsilon.\label{L.3-1}
\end{align}

Thus we can take the limit $\varepsilon\rightarrow0$ in (\ref{1.10}), (\ref{1.11}) and (\ref{L.3-1}) to conclude
\begin{eqnarray*}
&&\partial_t{v^i}+v\nabla v^i-\Delta {v^i}+\nabla_ip=-\nabla_j[\nabla_iu^kW_{p_j^k}(u,\nabla u)],\\
&&\nabla\cdot v=0
\end{eqnarray*}
and
\begin{align}
&\partial_tu\times u+(v\cdot\nabla)u\times u\nonumber\\
=&\nabla_\alpha[W_{p_\alpha}(u,\nabla u)]\times u-W_u(u,\nabla u)\times u.\label{L.3-2}
\end{align}
Recalling that $|u|=1$, one can calculate to get
\begin{eqnarray*}
&&-(\partial_tu\times u)\times u=(u\cdot u)\partial_tu-(\partial_tu\cdot u)u=\partial_tu,\\
&&-(v_k\nabla_ku\times u)\times u=v_k(u\cdot u)\nabla_ku-v_k(\nabla_ku\cdot u)u=(v\cdot\nabla)u,
\end{eqnarray*}
\begin{align*}
-\nabla_\alpha[W_{p_\alpha}(u,\nabla u)]\times u\times u=&(u\cdot u)\nabla_\alpha[W_{p_\alpha}(u,\nabla u)]-[\nabla_\alpha(W_{p_\alpha}(u,\nabla u))\cdot u]u\\
=&\nabla_\alpha[W_{p_\alpha}(u,\nabla u)]-\nabla_\alpha[(V_{p_\alpha}(u,\nabla u)\cdot u)u]\\
&+[V_{p_\alpha}(u,\nabla u)\cdot u]\nabla_\alpha u+[W_{p_\alpha}(u,\nabla u)\cdot\nabla_\alpha u]u,
\end{align*}
and
\begin{align*}
-W_u(u,\nabla u)\times u\times u=&(u\cdot u)W_u(u,\nabla u)-(W_u(u,\nabla u)\cdot u) u\\
=&W_u(u,\nabla u)-(W_u(u,\nabla u)\cdot u) u.
\end{align*}
By the aid of the above identities, we obtain
\begin{align*}
\partial_tu+&(v\cdot\nabla )u=\nabla_\alpha[W_{p_\alpha}(u,\nabla u)-(V_{p_\alpha}(u,\cdot\nabla u)\cdot u)u]-W_u(u,\nabla u)\\
+&(W_u(u,\nabla u)\cdot u)u+(W_{p_\alpha}(u,\nabla u)\cdot\nabla_\alpha u)u+(V_{p_\alpha}(u,\nabla u)\cdot u)\nabla_\alpha u,
\end{align*}
which is exactly (\ref{1.7}).

The uniqueness of strong solutions follows from the regularities stated in (\ref{2.17}) by using the standard argument. The proof is completed. \end{proof}


\section{Blow up criteria}\label{sec3}

In this section, we establish Serrin type or Beal-Kato-Majda type or mixed type (Serrin condition on one field and Beal-Kato-Majda condition on the other one) blow up criteria to strong solutions to the Ericksen-Leslie system, in other words, which will complete the proof of Theorem \ref{thm1} on the blow up criteria.

Strong solutions to the Ericksen-Leslie system satisfy the following basic energy balance law.

\begin{lemma}\label{lem4.0}
Let $(u, v)$ be a strong solution to (\ref{1.5})--(\ref{1.7}) in $\mathbb R^3\times(0, T)$. Then we have
\begin{align}
\frac{d}{dt}\int_{\mathbb R^3}\left(\frac{|v|^2}{2}+W(u,\nabla u)\right)dx+\int_{\mathbb R^3}(|\nabla v|^2+|\partial_tu+(v\cdot\nabla u)|^2dx=0.\label{7.1}
\end{align}
for any $t\in(0, T)$.
\end{lemma}

\begin{proof}
Equality (\ref{7.1}) follows from by multiplying (\ref{1.5}) by $v^i$ and (\ref{1.7}) by $\partial_tu^i+v\nabla u^i$, summing the resulting equations up and integrating over $\mathbb R^3$. Details can be found in the proof of Lemma 3.1 in \cite{HX}.\end{proof}

The following lemma states high order energy inequalities on the strong solutions to the Ericksen-Leslie system, which is one of the key lemmas of this paper.

\begin{lemma}\label{lem4.1}
Let $(u, v)$ be a strong solution to (\ref{1.5})--(\ref{1.7}) in $\mathbb R^3\times(0, T)$. Then
\begin{align}
&\frac{d}{dt}\int_{\mathbb R^3}(|\nabla v|^2+|\Delta u|^2)dx+\int_{
\mathbb R^3}(|\Delta v|^2+\frac{3a}{2}|\nabla^3u|^2)dx\nonumber\\
\leq&C\min\left\{\int_{\mathbb R^3}|v|^2|(|\nabla v|^2+\nabla^2u|^2)dx, \int_{\mathbb R^3}|\nabla v|(|\nabla v|^2+\nabla^2u|^2)dx\right\}\nonumber\\
&+C\int_{\mathbb R^3}|\nabla u|^2(|\nabla^2u|^2+|\nabla v|^2)dx
\end{align}
and
\begin{align}
&\frac{d}{dt}\int_{\mathbb R^3}(|\nabla^3u|^2+|\Delta v|^2)dx+\int_{\mathbb R^3}(\frac{3a}{2}|\nabla^4u|^2+|\nabla^3v|^2)dx\nonumber\\
\leq&C\int_{\mathbb R^3}[|\nabla u|^2(|\Delta v|^2+|\nabla^3u|^2)+|\nabla^2u|^4+|\nabla v|^4]dx\nonumber\\
&+C\min\left\{\int_{\mathbb R^3}|v|^2|\nabla^2v|^2dx,\int_{\mathbb R^3}|\nabla v||\nabla^2v|^2\right\}\label{3.3}
\end{align}
for any $t\in(0, T)$.
\end{lemma}

\begin{proof}

Differentiating (\ref{1.7}) with respect to $x_\beta$, multiplying the resulting equations by $-\nabla_\beta\Delta u^i$ and integrating over $\mathbb R^3$, we then obtain
\begin{align}
\frac{d}{dt}\int_{\mathbb R^3}&\frac{|\Delta u|^2}{2}dx=-\int_{\mathbb R^3}\nabla_{\alpha\beta}(W_{p_\alpha^i}(u,\nabla u))\nabla_\beta\Delta u^idx+\int_{\mathbb R^3}\nabla_\beta[\nabla_\alpha(u^ku^iV_{p_\alpha^k}(u,\nabla u)\nonumber\\
&+W_{u^i}(u,\nabla u)]\nabla_\beta\Delta u^i dx-\int_{\mathbb R^3}\nabla_\beta[u^ku^iW_{u^k}(u,\nabla u)+\nabla_\alpha u^kW_{p_\alpha^k}W_{p_\alpha^k}(u,\nabla u)u^i\nonumber\\
&+V_{p_\alpha^k}(u,\nabla u)u^k\nabla_\alpha u^i]\nabla_\beta\Delta u^idx+\int_{\mathbb R^3}\nabla_\beta(v\cdot\nabla u^i)\nabla_\beta\Delta u^idx. \label{4.1}
\end{align}
The terms on the right hand side of the above identity can be estimated term by term as follows. Estimates on all terms, except the last one, can be found in \cite{HX}. For reader's convenience, we rewrite them here. Similar to (\ref{2.5}), there holds
\begin{align}
&-\int_{\mathbb R^3}\nabla_{\alpha\beta}(W_{p_\alpha^i}(u,\nabla u))\nabla_\beta\Delta u^idx\nonumber\\
\leq&-(a-\eta)\int_{\mathbb R^3}|\nabla^3u|^2dx+C\int_{\mathbb R^3}(|\nabla u|^2|\nabla^2u|^2+|\nabla u|^6)dx.\label{4.2}
\end{align}
One checks that
\begin{align*}
\nabla_{\alpha_\beta}(u^ku^iV_{p_\alpha^k}&(u,\nabla u))
=\nabla_{\alpha\beta}(u^ku^i)V_{p_\alpha^k}(u,\nabla u)+\nabla_\alpha(u^ku^i)\nabla_\beta(V_{p_\alpha^k}(u,\nabla u))\\
&+\nabla_\beta(u^ku^i)\nabla_\alpha(V_{p_\alpha^k}(u,\nabla u))+u^ku^i\nabla_{\alpha\beta}(V_{p_\alpha^k}(u,\nabla u))\\
=&\nabla_{\alpha\beta}(u^ku^i)V_{p_\alpha^k}(u,\nabla u)+\nabla_\alpha(u^ku^i)\nabla_\beta(V_{p_\alpha^k}(u,\nabla u))\\
&+\nabla_\beta(u^ku^i)\nabla_\alpha(V_{p_\alpha^k}(u,\nabla u))+u^ku^i\nabla_{\alpha}(V_{u^jp_\alpha^k}(u,\nabla u)\nabla_\beta u^j)\\
&+u^ku^i\nabla_\alpha(V_{p_\alpha^kp_\sigma^j}(u,\nabla u))\nabla_{\beta\sigma}u^j+u^ku^iV_{p_\alpha^kp_\sigma^j}(u,\nabla u)\nabla_{\alpha\beta\sigma}u^j,
\end{align*}
which implies
\begin{align}
&\int_{\mathbb R^3}\nabla_{\alpha\beta}(u^ku^iV_{p_\alpha^k}(u,\nabla u))\nabla_\beta\Delta u^idx\nonumber\\
\leq&\int_{\mathbb R^3}u^ku^iV_{p_\alpha^kp_\sigma^j}(u,\nabla u)\nabla_{\alpha\beta\sigma}u^j\nabla_\beta\Delta u^idx\nonumber\\
&+C\int_{\mathbb R^3}(|\nabla u||\nabla^2u|+|\nabla u|^3)|\nabla^3u|dx\nonumber\\
\leq&C\int_{\mathbb R^3}[(|\nabla u||\nabla^2u|+|\nabla u|^3)|\nabla^3u|+|u^i\nabla_\beta\Delta u^i||\nabla^3u|]dx\nonumber\\
=&C\int_{\mathbb R^3}[(|\nabla u||\nabla^2u|+|\nabla u|^3)|\nabla^3u|+|\nabla_\beta(u^i\Delta u^i)-\nabla_\beta u^i\Delta u^i||\nabla^3u|]dx\nonumber\\
=&C\int_{\mathbb R^3}[(|\nabla u||\nabla^2u|+|\nabla u|^3)|\nabla^3u|+|\nabla_\beta(|\nabla u|^2+\nabla_\beta u^i\Delta u^i||\nabla^3u|]dx\nonumber\\
\leq&\eta\int_{\mathbb R^3}|\nabla^3u|^2dx+C\int_{\mathbb R^3}(|\nabla u|^2|\nabla^2u|^2+|\nabla u|^6)dx. \label{4.3}
\end{align}
Here, we have used the fact that $\Delta u\cdot u=-|\nabla u|^2$ guaranteed by $|u|=1$.
One can check easily that
\begin{align*}
&|\nabla_\beta[W_{u^i}(u,\nabla u)-u^ku^iW_{u^k}(u,\nabla u)\\
&-\nabla_\alpha u^kW_{p_\alpha^k}(u,\nabla u)u^i-V_{p_\alpha^k}(u,\nabla u)u^k\nabla_\alpha u^i]|\\
\leq&C(|\nabla u||\nabla^2u|+|\nabla u|^3).
\end{align*}
Thus
\begin{align}
&\int_{\mathbb R^3}\nabla_\beta[W_{u^i}(u,\nabla u)-u^ku^iW_{u^k}(u,\nabla u)\nonumber\\
&-\nabla_\alpha u^kW_{p_\alpha^k}(u,\nabla u)u^i-V_{p_\alpha^k}(u,\nabla u)u^k\nabla_\alpha u^i]]\nabla_\beta\Delta u^i dx\nonumber\\
\leq&\eta \int_{\mathbb R^3}|\nabla^3 u|^2dx+C\int_{\mathbb R^3}(|\nabla u|^6+|\nabla u|^2|\nabla^2u|^2)dx.\label{4.4}
\end{align}
It follows that
\begin{align*}
&\int_{\mathbb R^3}\nabla_\beta(v\cdot\nabla u^i)\nabla_\beta\Delta u^idx\\
=&\int_{\mathbb R^3}\nabla_\beta v\cdot\nabla u^i\nabla_\beta\Delta u^idx+\int_{\mathbb R^3}v\cdot\nabla\nabla_\beta v\nabla_\beta\Delta u^idx\\
=&\int_{\mathbb R^3}\nabla_\beta v\cdot\nabla u^i\nabla_\beta\Delta u^idx-\int_{\mathbb R^3}(\nabla_\alpha v\cdot\nabla\nabla_\beta u^i\nabla_{\alpha\beta}u^i+v\cdot\nabla\nabla_{\alpha\beta}u^i\nabla_{\alpha\beta}u^i)dx\\
=&\int_{\mathbb R^3}(\nabla_\beta v\cdot\nabla u^i\nabla_\beta\Delta u^i-\nabla_\alpha v\cdot\nabla\nabla_\beta u^i\nabla_{\alpha\beta}u^i)dx\\
\leq&C\int_{\mathbb R^3}(|\nabla v||\nabla u||\nabla^3u|+|\nabla v||\nabla^2u|^2)dx\\
\leq&\eta\int_{\mathbb R^3}|\nabla^3u|^2dx+C\int_{\mathbb R^3}(|\nabla v|^2|\nabla u|^2+|\nabla v||\nabla^2u|^2)dx,
\end{align*}
and
\begin{align*}
&\int_{\mathbb R^3}\nabla_\beta(v\cdot\nabla u^i)\nabla_\beta\Delta u^idx=\int_{\mathbb R^3}(\nabla_\beta v\cdot\nabla u^i+v\cdot\nabla\nabla_\beta u^i)\nabla_\beta\Delta u^idx\\
\leq&\eta\int_{\mathbb R^3}|\nabla^3u|^2dx+C\int_{\mathbb R^3}(|\nabla v|^2|\nabla u|^2+|v|^2|\nabla^2u|^2)dx.
\end{align*}
Hence, it holds that
\begin{align}
\int_{\mathbb R^3}\nabla_\beta(v\cdot\nabla u^i)\nabla_\beta\Delta u^idx
\leq& C\int_{\mathbb R^3}|\nabla u|^2|\nabla v|^2dx+\eta\int_{\mathbb R^3}|\nabla^3u|^2dx\nonumber\\
&+C\min\left\{\int_{\mathbb R^3}|v|^2|\nabla^2u|^2dx, \int_{\mathbb R^3}|\nabla v||\nabla^2u|^2dx\right\}.\label{4.5}
\end{align}
Substituting (\ref{4.2})--(\ref{4.5}) into (\ref{4.1}) yields
\begin{align}
&\frac{d}{dt}\int_{\mathbb R^3}|\Delta u|^2dx+a\int_{\mathbb R^3}|\nabla^3u|^2dx\nonumber\\
\leq&C\int_{\mathbb R^3}|\nabla u|^2(|\nabla u|^4+|\nabla^2u|^2+|\nabla v|^2)dx\nonumber\\
&+C\min\left\{\int_{\mathbb R^3}|v|^2|\nabla^2u|^2dx, \int_{\mathbb R^3}|\nabla v||\nabla^2u|^2dx\right\}.\label{4.6}
\end{align}

Multiplying (\ref{1.5}) by $-\Delta v^i$ and integrating the resulting equation over $\mathbb R^3$ yields
\begin{align*}
\frac{d}{dt}\int_{\mathbb R^3}\frac{|\nabla v|^2}{2}dx+\int_{\mathbb R^3}|\Delta v|^2dx=\int_{\mathbb R^3}(v\cdot\nabla v^i+\nabla_j(\nabla_iu^kW_{p_j^k}(u,\nabla u))\Delta v^idx.
\end{align*}
It follows that
\begin{align*}
\int_{\mathbb R^3}v\cdot\nabla v^i\Delta v^idx\leq\eta\int_{\mathbb R^3}|\Delta v|^2dx+C\int_{\mathbb R^3}|v|^2|\nabla v|^2dx,
\end{align*}
and
\begin{align*}
\int_{\mathbb R^3}v\cdot\nabla v^i\Delta v^idx=-\int_{\mathbb R^3}(\nabla_\beta v\cdot\nabla v^i\nabla_\beta v^i+v\cdot\nabla\nabla_\beta v^i\nabla_\beta v^i)dx\leq\int_{\mathbb R^3}|\nabla v|^3dx.
\end{align*}
Since $|\nabla_j(\nabla_iu^kW_{p_j^k}(u,\nabla u)|\leq C(|\nabla u||\nabla^2u|+|\nabla u|^3$, it follows that
$$
\int_{\mathbb R^3}\nabla_j\nabla_iu^kW_{p_j^k}(u,\nabla u)\Delta v^idx\leq\eta\int_{\mathbb R^3}|\Delta v|^2dx+C\int_{\mathbb R^3}|\nabla u|^2(|\nabla u|^4+|\nabla^2u|^2)dx.
$$
Hence
\begin{align*}
&\frac{d}{dt}\int_{\mathbb R^3}|\nabla v|^2dx+\int_{\mathbb R^3}|\Delta v|^2dx\nonumber\\
\leq& C\int_{\mathbb R^3}|\nabla u|^2(|\nabla u|^4+|\nabla^2u|^2)dx+C\min\left\{\int_{\mathbb R^3}|v|^2|\nabla v|^2dx,\int_{\mathbb R^3}|\nabla v|^3dx\right\}.
\end{align*}
Combining this with (\ref{4.6}) gives
\begin{align}
&\frac{d}{dt}\int_{\mathbb R^3}(|\nabla v|^2+|\Delta u|^2)dx+\int_{
\mathbb R^3}(|\Delta v|^2+|\nabla^3u|^2)dx\nonumber\\
\leq&C\min\left\{\int_{\mathbb R^3}|v|^2|(|\nabla v|^2+\nabla^2u|^2)dx, \int_{\mathbb R^3}|\nabla v|(|\nabla v|^2+|\nabla^2u|^2)dx\right\}\nonumber\\
&+C\int_{\mathbb R^3}|\nabla u|^2(|\nabla u|^4+|\nabla^2u|^2+|\nabla v|^2)dx\nonumber\\
\leq&C\min\left\{\int_{\mathbb R^3}|v|^2|(|\nabla v|^2+\nabla^2u|^2)dx, \int_{\mathbb R^3}|\nabla v|(|\nabla v|^2+|\nabla^2u|^2)dx\right\}\nonumber\\
&+C\int_{\mathbb R^3}|\nabla u|^2(|\nabla^2u|^2+|\nabla v|^2)dx.\label{4.7}
\end{align}
In the last step of the above inequality, we have used the fact that $|\nabla u|^2=-\Delta u\cdot u$.

Now we prove (\ref{3.3}). Multiplying (\ref{1.5}) by $\Delta^2v^i$ and integrating the resulting equation over $\mathbb R^3$ yield
\begin{align}
\frac{d}{dt}\int_{\mathbb R^3}\frac{|\Delta v|^2}{2}dx+\int_{\mathbb R^3}|\nabla^3v|^2dx=\int_{\mathbb R^3}[\nabla_j(\nabla_iu^kW_{p_j^k}(u,\nabla u))-v\cdot\nabla v^i]\Delta^2v^idx.\label{4.8}
\end{align}
It follows that
\begin{align}
-\int_{\mathbb R^3}v\cdot\nabla v^i\Delta^2v^idx=&-\int_{\mathbb R^3}(\Delta v\cdot\nabla v^i+2\nabla v\cdot\nabla^2v^i+v\nabla\Delta v^i)\Delta v^idx\nonumber\\
\leq& C\int_{\mathbb R^3}|\nabla v||\nabla^2v|^2,\label{4.9}
\end{align}
and
\begin{align}
-\int_{\mathbb R^3}v\cdot\nabla v^i\Delta^2v^idx=&\int_{\mathbb R^3}(\nabla_jv\cdot\nabla v^i+v\cdot\nabla\nabla_jv^i)\nabla_j\Delta v^idx\nonumber\\
\leq&\eta\int_{\mathbb R^3}|\nabla^3v|^2dx+C\int_{\mathbb R^3}(|\nabla v|^4+|v|^2|\nabla^2v|^2dx. \label{4.10}
\end{align}
Notice that
\begin{align*}
\int_{\mathbb R^3}|\nabla v|^4dx=&-\int_{\mathbb R^3}\textmd{div}(|\nabla v|^2\nabla v)vdx\leq C\int_{\mathbb R^3}|v||\nabla v|^2|\nabla^2v|dx\\
\leq&\frac{1}{2}\int_{\mathbb R^3}|\nabla v|^4dx+C\int_{\mathbb R^3}|v|^2|\nabla^2 v|^2dx,
\end{align*}
which implies
$$
\int_{\mathbb R^3}|\nabla v|^4dx\leq C\int_{\mathbb R^3}|v|^2|\nabla^2v|^2dx.
$$
This, together with (\ref{4.9})--(\ref{4.10}), shows
\begin{align}
-\int_{\mathbb R^3}v\cdot\nabla v^i\Delta^2v^idx\leq\eta\int_{\mathbb R^3}|\nabla^3v|^2dx+C\min\left\{\int_{\mathbb R^3}|v|^2|\nabla^2v|^2dx,\int_{\mathbb R^3}|\nabla v||\nabla^2v|^2\right\}.\label{4.11}
\end{align}
Note that
\begin{align*}
|\nabla_{jl}[\nabla_iu^kW_{p_j^k}(u,\nabla u)]|\leq& C(|\nabla u||\nabla^3u|+|\nabla u|^2|\nabla^2u|+|\nabla^2u|^2+|\nabla u|^4)\\
\leq&C(|\nabla u||\nabla^3u|+|\nabla^2u|^2),
\end{align*}
where we have used $|\nabla u|^2\leq|\Delta u|$. Then
\begin{align}
&\int_{\mathbb R^3}\nabla_j(\nabla_iu^kW_{p_j^k}(u,\nabla u))\Delta^2v^idx\nonumber\\
=&-\int_{\mathbb R^3}\nabla_{jl}(\nabla_iu^k W_{p_j^k}(u,\nabla u))\nabla_l\Delta v^idx\nonumber\\
\leq&C\int_{\mathbb R^3}(|\nabla u||\nabla^3u|+|\nabla^2u|^2)|\nabla^3v|dx\nonumber\\
\leq&\eta\int_{\mathbb R^3}|\nabla^3v|^2dx+C\int_{\mathbb R^3}(|\nabla u|^2|\nabla^3u|^2+|\nabla^2u|^4)dx.\label{4.12}
\end{align}
It follows that
\begin{align*}
\int_{\mathbb R^3}|\nabla^2u|^4dx=&-\int_{\mathbb R^3}\nabla_j(|\nabla^2u|^2\nabla_{ij}u)\nabla_iudx\leq\int_{\mathbb R^3}|\nabla u||\nabla^2u|^2|\nabla^3u|dx\\
\leq&\frac{1}{2}\int_{\mathbb R^3}|\nabla^2u|^2dx+C\int_{\mathbb R^3}|\nabla u|^2|\nabla^3u|^2dx,
\end{align*}
which, together with (\ref{4.12}), gives
\begin{align*}
\int_{\mathbb R^3}\nabla_j(\nabla_iu^kW_{p_j^k}(u,\nabla u))\Delta^2v^idx
\leq\eta\int_{\mathbb R^3}|\nabla^3v|^2dx+C\int_{\mathbb R^3}|\nabla u|^2|\nabla^3u|^2dx.
\end{align*}
Substitute the above inequality and (\ref{4.11}) into (\ref{4.8}) to get
\begin{align}
&\frac{d}{dt}\int_{\mathbb R^3}|\Delta v|^2dx+\int_{\mathbb R^3}|\nabla^3v|^2dx\nonumber\\
\leq&C\int_{\mathbb R^3}|\nabla u|^2|\nabla^3u|^2dx+C\min\left\{\int_{\mathbb R^3}|v|^2|\nabla^2v|^2dx,\int_{\mathbb R^3}|\nabla v||\nabla^2v|^2\right\}. \label{4.13}
\end{align}

Multiplying (\ref{1.7}) by $-\Delta^3u^i$ and integrating over $\mathbb R^3$ lead to
\begin{align}
\frac{d}{dt}\int_{\mathbb R^3}&\frac{|\nabla^3u|^2}{2}dx=-\int_{\mathbb R^3}\nabla_\alpha(W_{p_\alpha^i}(u,\nabla u))\Delta^3u^idx+\int_{\mathbb R^3}[\nabla_\alpha(u^ku^iV_{p_\alpha^k}(u,\nabla u))\nonumber\\
&+W_{u^i}(u,\nabla u)]\Delta^3u^idx-\int_{\mathbb R^3}[W_u^k(u,\nabla u)u^ku^i+W_{p_\alpha^k}(u,\nabla u)\nabla_\alpha u^ku^i\nonumber\\
&+V_{p_\alpha^k}(u,\nabla u)u^k\nabla_\alpha u^i]\Delta^3u^idx+\int_{\mathbb R^3}v\nabla u^i\Delta^3u^idx. \label{4.14}
\end{align}
Direct calculations give
\begin{align*}
&\nabla_{\beta\gamma\sigma}(W_{p_\alpha^i}(u,\nabla u))=\nabla_{\beta\gamma}[W_{u^jp_\alpha^i}(u,\nabla u)\nabla_\sigma u^j+W_{p_\alpha^ip_l^j}(u,\nabla u)\nabla_{\sigma l}u^j]\\
=&\nabla_{\beta\gamma}[W_{u^jp_\alpha^i}(u,\nabla u)\nabla_\sigma u^j]+\nabla_{\beta\gamma}[W_{p_\alpha^ip_l^j}(u,\nabla u)]\nabla_{\sigma l}u^j+\nabla_{\beta}[W_{p_\alpha^ip_l^j}(u,\nabla u)]\nabla_{\gamma\sigma l}u^j\\
&+\nabla_{\gamma}[W_{p_\alpha^ip_l^j}(u,\nabla u)]\nabla_{\beta\sigma l}u^j+W_{p_\alpha^ip_l^j}(u,\nabla u)\nabla_{\beta\gamma\sigma l}u^j.
\end{align*}
Due to $W_{p_\alpha^ip_\beta^j}(z,p)\xi_\alpha^i\xi_\beta^j\geq a|\xi|^2$ and
\begin{align*}
&|\nabla_{\beta\gamma}[W_{u^jp_\alpha^i}(u,\nabla u)\nabla_\sigma u^j]+\nabla_{\beta\gamma}[W_{p_\alpha^ip_l^j}(u,\nabla u)]\nabla_{\sigma l}u^j\\
&+\nabla_{\beta}[W_{p_\alpha^ip_l^j}(u,\nabla u)]\nabla_{\gamma\sigma l}u^j+\nabla_{\gamma}[W_{p_\alpha^ip_l^j}(u,\nabla u)]\nabla_{\beta\sigma l}u^j|\\
\leq& C(|\nabla u||\nabla^3u|+|\nabla u|^2|\nabla^2u|+|\nabla^2 u|^2)\leq C(|\nabla u||\nabla^3u|+|\nabla^2 u|^2),
\end{align*}
one can get
\begin{align}
&-\int_{\mathbb R^3}\nabla_\alpha(W_{p_\alpha^i}(u,\nabla u))\Delta^3u^idx=-\int_{\mathbb R^3}\nabla_{\beta\gamma\sigma}(W_{p_\alpha^i}(u,\nabla u))\nabla_{\alpha\beta\gamma\sigma}u^idx\nonumber\\
\leq&-a\int_{\mathbb R^3}|\nabla^4u|^2dx+C\int_{\mathbb R^3}(|\nabla u||\nabla^3u|+|\nabla^2 u|^2)|\nabla^4u|dx\nonumber\\
\leq&-(a-\eta)\int_{\mathbb R^3}|\nabla^4u|^2dx+C\int_{\mathbb R^3}(|\nabla u|^2|\nabla^3u|^2+|\nabla^2 u|^4)dx.\label{4.15}
\end{align}
Next
\begin{align*}
\nabla_\alpha\Delta&[u^ku^iV_{p_\alpha^k}(u,\nabla u)]=\nabla_\alpha[\Delta(u^ku^i)V_{p_\alpha^k}(u,\nabla u)\\
&+2\nabla(u^ku^i)\nabla(V_{p_\alpha^k}(u,\nabla u))+u^ku^i\Delta(V_{p_\alpha^k}(u,\nabla u)]\\
=&\nabla_\alpha\Delta(u^ku^i)V_{p_\alpha^k}(u,\nabla u)+\Delta(u^ku^i)\nabla_\alpha[V_{p_\alpha^k}(u,\nabla u)]\\
&+2\nabla_\alpha\nabla(u^ku^i)\nabla(V_{p_\alpha^k}(u,\nabla u))+2\nabla(u^ku^i)\nabla_\alpha\nabla(V_{p_\alpha^k}(u,\nabla u))\\
&+\nabla_\alpha(u^ku^i)\Delta(V_{p_\alpha^k}(u,\nabla u))+u^ku^i\Delta\nabla_\alpha(V_{p_\alpha^k}(u,\nabla u)).
\end{align*}
Note that
\begin{align*}
&|\nabla_\alpha\Delta(u^ku^i)V_{p_\alpha^k}(u,\nabla u)+\Delta(u^ku^i)\nabla_\alpha[V_{p_\alpha^k}(u,\nabla u)]\\
&+2\nabla_\alpha\nabla(u^ku^i)\nabla(V_{p_\alpha^k}(u,\nabla u))+2\nabla(u^ku^i)\nabla_\alpha\nabla(V_{p_\alpha^k}(u,\nabla u))\\
&+\nabla_\alpha(u^ku^i)\Delta(V_{p_\alpha^k}(u,\nabla u))|\\
\leq&C[(|\nabla u||\nabla^2u|+|\nabla^3u|)|\nabla u|+(|\nabla u|^2+|\nabla^2u|)(|\nabla u|^2+|\nabla^2u|)\\
&+|\nabla u|(|\nabla^3u|+|\nabla u||\nabla^2u|+|\nabla u|^3)\\
\leq&C(|\nabla u||\nabla^3u|+|\nabla u|^2|\nabla^2u|+|\nabla^2u|^2+|\nabla u|^4)\\
\leq&C(|\nabla u||\nabla^3u|+|\nabla^2u|^2)
\end{align*}
and
\begin{align*}
|\Delta\nabla_\alpha(V_{p_\alpha^k}(u,\nabla u))|\leq& C(|\nabla^4u|+|\nabla^3u||\nabla u|+|\nabla^2u|^2+|\nabla^2u||\nabla u|^2)\\
\leq&C(|\nabla^4u|+|\nabla^3u||\nabla u|+|\nabla^2u|^2),
\end{align*}
where we have used again $|\nabla u|^2=-\Delta u\cdot u$. Hence
\begin{align*}
&\int_{\mathbb R^3}\nabla_\alpha[u^ku^iV_{p_\alpha^k}(u,\nabla u)]\Delta^3u^idx=\int_{\mathbb R^3}\nabla_\alpha\Delta[u^ku^iV_{p_\alpha^k}(u,\nabla u)]\Delta^2u^idx\\
\leq&C\int_{\mathbb R^3}(|\nabla u||\nabla^3u|+|\nabla^2u|^2)|\Delta^2u|dx+\int_{\mathbb R^3}u^ku^i\Delta^2u^i\Delta\nabla_\alpha(V_{p_\alpha^k}(u,\nabla u))dx.
\end{align*}
Since
\begin{align*}
|u^i\Delta^2u^i|=&|\Delta(u^i\Delta u^i)-2\nabla u^i\nabla\Delta u^i-|\Delta u^i|^2|\\
=&|\Delta(|\nabla u|^2)+2\nabla u^i\nabla\Delta u^i+|\Delta u^i|^2|\leq C(|\nabla u||\nabla^3u|+|\nabla^2u|^2),
\end{align*}
we arrive at
\begin{align}
&\int_{\mathbb R^3}\nabla_\alpha[u^ku^iV_{p_\alpha^k}(u,\nabla u)]\Delta^3u^idx
\leq C\int_{\mathbb R^3}[(|\nabla u||\nabla^3u|+|\nabla^2u|^2)|\Delta^2u|\nonumber\\
&+(|\nabla^4u|+|\nabla^3u||\nabla u|+|\nabla^2u|^2)(|\nabla u||\nabla^3u|+|\nabla^2u|^2)]dx\nonumber\\
\leq&\eta\int_{\mathbb R^3}|\nabla^4u|^2dx+C\int_{\mathbb R^3}(|\nabla u|^2|\nabla^3u|^2+|\nabla^2u|^4)dx. \label{4.16}
\end{align}
One can check that
\begin{align*}
&|\Delta[W_{u^i}(u,\nabla u)-W_{u^k}(u,\nabla u)u^ku^i-W_{p_\alpha^k}(u,\nabla u)\nabla_\alpha u^ku^i-V_{p_\alpha^k}(u,\nabla u)u^k\nabla_\alpha u^i]\\
\leq&C(|\nabla u||\nabla^3u|+|\nabla u|^2|\nabla^2u|+|\nabla^2u|^2+|\nabla u|^4)\leq C(|\nabla u||\nabla^3u|+|\nabla^2u|^2),
\end{align*}
which implies
\begin{align}
&\int_{\mathbb R^3}[W_{u^i}(u,\nabla u)-W_{u^k}(u,\nabla u)u^ku^i\nonumber\\
&-W_{p_\alpha^k}(u,\nabla u)\nabla_\alpha u^ku^i-V_{p_\alpha^k}(u,\nabla u)u^k\nabla_\alpha u^i]\Delta^3u^idx\nonumber\\
=&\int_{\mathbb R^3}\Delta[W_{u^i}(u,\nabla u)-W_{u^k}(u,\nabla u)u^ku^i\nonumber\\
&-W_{p_\alpha^k}(u,\nabla u)\nabla_\alpha u^ku^i-V_{p_\alpha^k}(u,\nabla u)u^k\nabla_\alpha u^i]\Delta^2u^idx\nonumber\\
\leq&C\int_{\mathbb R^3}(|\nabla u||\nabla^3u|+|\nabla^2u|^2)|\nabla^4u|dx\nonumber\\
\leq&\eta\int_{\mathbb R^3}|\nabla^4u|^2dx+C\int_{\mathbb R^3}(|\nabla u|^2|\nabla^3u|^2+|\nabla^2u|^4)dx.\label{4.17}
\end{align}
Integrating by parts gives
\begin{align}
&\int_{\mathbb R^3}v\cdot\nabla u^i\Delta^3u^idx=\int_{\mathbb R^3}(\Delta v\cdot\nabla u^i+2\nabla v\cdot\nabla^2u^i+v\cdot\nabla\Delta u^i)\Delta^2u^idx\nonumber\\
=&\int_{\mathbb R^3}[(\Delta v\cdot\nabla u^i+2\nabla v\cdot\nabla^2u^i)\Delta^2 u^i-\nabla_jv\cdot\nabla\Delta u^i\nabla_j\Delta u^i-v\cdot\nabla\Delta\nabla_ju^i\Delta\nabla_ju^i]dx\nonumber\\
=&\int_{\mathbb R^3}[(\Delta v\cdot\nabla u^i+2\nabla v\cdot\nabla^2u^i)\Delta^2 u^i-\nabla_jv\cdot\nabla\Delta u^i\nabla_j\Delta u^i]dx\nonumber\\
=&\int_{\mathbb R^3}[(\Delta v\cdot\nabla u^i+2\nabla v\cdot\nabla^2u^i)\Delta^2 u^i+\Delta u^i\nabla_jv\cdot\nabla_j\nabla\Delta u^i]dx\nonumber\\
\leq&\eta\int_{\mathbb R^3}|\nabla^4u|^2dx+C\int_{\mathbb R^3}(|\nabla u|^2|\Delta v|^2+|\nabla v|^2|\nabla^2u|^2)dx\nonumber\\
\leq&\eta\int_{\mathbb R^3}|\nabla^4u|^2dx+C\int_{\mathbb R^3}(|\nabla u|^2|\Delta v|^2+|\nabla^2u|^4+|\nabla v|^4)dx. \label{4.18}
\end{align}
Substituting (\ref{4.15})--(\ref{4.18}) into (\ref{4.14}) yields
\begin{align*}
&\frac{d}{dt}\int_{\mathbb R^3}|\nabla^3u|^2dx+a\int_{\mathbb R^3}|\nabla^4u|^2dx\nonumber\\
\leq& C\int_{\mathbb R^3}[|\nabla u|^2(|\Delta v|^2+|\nabla^3u|^2)+|\nabla^2u|^4+|\nabla v|^4]dx,
\end{align*}
which, combined with (\ref{4.13}), gives
\begin{align*}
&\frac{d}{dt}\int_{\mathbb R^3}(|\nabla^3u|^2+|\Delta v|^2)dx+\int_{\mathbb R^3}(a|\nabla^4u|^2+|\nabla^3v|^2)dx\nonumber\\
\leq&C\int_{\mathbb R^3}[|\nabla u|^2(|\Delta v|^2+|\nabla^3u|^2)+|\nabla^2u|^4+|\nabla v|^4]dx\nonumber\\
&+C\min\left\{\int_{\mathbb R^3}|v|^2|\nabla^2v|^2dx,\int_{\mathbb R^3}|\nabla v||\nabla^2v|^2\right\}.
\end{align*}
This proves (\ref{3.3}).
\end{proof}

We also need the following logarithmic type Sobolev inequality to control $L^\infty$ norm of $\nabla v$ in term of its $BMO$ and higher order norms.

\begin{lemma}\label{lem4.2}
For any $f\in L^1(0,T; BMO(\mathbb R^3))\cap L^1(0, T; L^q(\mathbb R^3))$ and $\nabla f\in L^1$ $(0, T; L^p(\mathbb R^3))$ with $p\in(3,\infty)$ and $q\in[1,\infty)$, it holds that
\begin{align*}
\int_s^t\|f(\tau)\|_{L^\infty}d\tau\leq &C\left[\int_s^t[f(\tau)]_{BMO}d\tau\ln\left(1+\int_s^t\|\nabla f(\tau)\|_{L^p}d\tau\right)\right.\\
&\left.+\int_s^t\|f\|_{L^q}d\tau+1\right],
\end{align*}
with $C$ being a positive constant depending only on $p$ and $q$.
\end{lemma}

\begin{proof}
Set
$$
f_r(x)=\frac{1}{|B_r|}\int_{B_r(x)}f(y)dy.
$$
For any $r\geq 1$, we apply H\"older's inequality to obtain
\begin{align*}
|f_r(x)|\leq \frac{C}{r^3}\int_{B_r(x)}|f(y)|dy\leq C\left(\frac{1}{r^3}\int_{B_r(x)}|f(y)|^qdy\right)^{1/q}\leq C \|f\|_{L^q(\mathbb R^3)}.
\end{align*}
For any $r<1$, there is a unique $k\in\mathbb N$, such that
$$
\frac{1}{2^{k}}\leq r<\frac{1}{2^{k-1}},\quad k\leq C(1+|\ln r|),
$$
and thus
\begin{align*}
|f_r(x)|\leq&\sum_{j=1}^k|f_{2^{j-1}r}(x)-f_{2^jr}(x)|+|f_{2^kr}(x)|\\
\leq&\sum_{j=1}^k\frac{1}{|B_{2^{j-1}r}|}\int_{B_{2^{j-1}r}(x)}|f(y)-f_{2^jr(x)}|dy+C\|f\|_{L^q(\mathbb R^3)}\\
\leq&\sum_{j=1}^k\frac{C}{|B_{2^{j}r}|}\int_{B_{2^{j}r}(x)}|f(y)-f_{2^jr(x)}|dy+C\|f\|_{L^q(\mathbb R^3)}\\
\leq&Ck[f]_{BMO(\mathbb R^3)}+C \|f\|_{L^q(\mathbb R^3)}\\
\leq&C(1+|\ln r|)[f]_{BMO(\mathbb R^3)}+C\|f\|_{L^q(\mathbb R^3)}.
\end{align*}
By a variant of the Sobolev embedding theorem (see e.g. page 268 of \cite{EVANS}), the above inequalities give
\begin{align*}
|f(x)|\leq&|f(x)-f_r(x)|+|f_r(x)|\\
\leq&Cr^{1-3/p}\|\nabla f\|_{L^p(\mathbb R^3)}+C(1+|\ln r|)[f]_{BMO(\mathbb R^3)}+C\|f\|_{L^q(\mathbb R^3)}
\end{align*}
for any $r<1$.
Integrating the above inequality over $(s, t)$ yields
\begin{align*}
\int_s^t\|f\|_{L^\infty}d\tau\leq &Cr^{1-3/p}\int_s^t\|\nabla f\|_{L^p(\mathbb R^3)}d\tau+C(1+|\ln r|)\int_s^t[f]_{BMO(\mathbb R^3)}d\tau\\
&+\int_s^t\|f\|_{L^q(\mathbb R^3)d\tau}
\end{align*}
for any $r<1$. Taking $r=\left(1+\int_s^t\|\nabla f\|_{L^p(\mathbb R^3)}d\tau\right)^{-p/(p-3)}$ in the above inequality proves the lemma.
\end{proof}

Now we finish the proof of the blow up criteria in Theorem \ref{thm1}.

 \begin {proof}  [Proof of the blow up criteria in Theorem \ref{thm1}]
Let $T^*$ be the maximum existence time for the strong solution $(u, v)$ to the system (\ref{1.5})--(\ref{1.7}). Suppose, by contradiction, that the conclusion fails. Then both the following two hold true
\begin{itemize}
  \item $\|\Delta u\|_{L^1(0,T^*; L^\infty(\mathbb R^3))}<\infty$ or $\|\nabla u\|_{L^{q_1}(0,T^*;L^{r_1}(\mathbb R^3))}<\infty$
for some $q_1\in[2,\infty), r_1\in(3,\infty)$ with $\frac{2}{q_1}+\frac{3}{r_1}=1$;
  \item $\|\omega\|_{L^1(0,T^*; BMO(\mathbb R^3))}<\infty$ or $\|v\|_{L^q(0,T^*; L^r(\mathbb R^3))}<\infty$ for some $q_2\in[2,\infty), r_2\in(3,\infty)$ with $\frac{2}{q_2}+\frac{3}{r_2}=1$.
\end{itemize}

By the Sobolev embedding inequality and the H\"older inequality, there holds
\begin{align}
\int_{\mathbb R^3}|w|^2|f|^2dx\leq \|w\|_{L^r}^2\|f\|_{L^2}^{\frac{2r-6}{r}}\|f\|_{L^6}^{\frac{6}{r}}\leq\eta\|\nabla f\|_{L^2}^2+C\|w\|_{L^r}^q\|f\|_{L^2}^2\label{4.19}
\end{align}
for any $r\in(3,\infty], q\in[2,\infty)$ with $\frac{2}{q}+\frac{3}{r}=1$. Note that $[\nabla v]_{BMO(\mathbb R^3)}\leq[\omega]_{BMO(\mathbb R^3)}$ by elliptic estimates. By the aid of this inequality, (\ref{4.19}), and $|\nabla u|^2\leq|\Delta u|$, one can get from Lemma \ref{lem4.1} and Lemma \ref{lem4.2} that
\begin{align}
&\left[\int_{\mathbb R^3}(|\nabla v|^2+|\Delta u|^2)dx\right](t)+\int_s^t\int_{\mathbb R^3}(|\Delta v|^2+a|\nabla^3u|^2)dxd\tau\nonumber\\
\leq&\left[\int_{\mathbb R^3}(|\nabla v|^2+|\Delta u|^2)dx\right](s)\exp\left\{C\min\left\{\int_s^t\|v\|_{L^{r_2}}^{q_2}d\tau,\int_s^t\|\nabla v\|_{L^\infty}d\tau\right\}\right\}\nonumber\\
&\times\exp\left\{C\min\left\{\int_s^t\|\nabla u\|_{L^{r_1}}^{q_1}d\tau,\int_s^t\|\Delta u\|_{L^\infty}d\tau\right\}\right\}\nonumber\\
\leq&C\exp\left\{C\min\left\{\int_s^t\|v\|_{L^{r_2}}^{q_2}d\tau,\int_s^t[\nabla v]_{BMO}d\tau\ln\left(1+\int_s^t\|\nabla^3v\|_{L^2}d\tau\right)\right.\right.\nonumber\\
&\left.\left.+\int_s^t\|\nabla v\|_{L^2(\mathbb R^3)}^2d\tau\right\}\right\}\left[\int_{\mathbb R^3}(|\nabla v|^2+|\Delta u|^2)dx\right](s)\nonumber\\
\leq&C\exp\left\{C\min\left\{\int_s^t\|v\|_{L^{r_2}}^{q_2}d\tau,\int_s^t[\omega]_{BMO}d\tau\ln\left(1+\int_s^t\|\nabla^3v\|_{L^2}d\tau\right)\right\}\right\}\nonumber\\
&\times\left[\int_{\mathbb R^3}(|\nabla v|^2+|\Delta u|^2)dx\right](s)\label{4.20}
\end{align}
for any $0<s\leq t<T^*$.

If $v\in L^{q_2}(\frac{T^*}{2}, T^*; L^{r_2})$ for some $q_2\in[2,\infty), r_2\in(3,\infty]$ with $\frac{2}{q_2}+\frac{3}{r_2}=1$, then the above inequality  shows that
\begin{align*}
&\left[\int_{\mathbb R^3}(|\nabla v|^2+|\Delta u|^2)dx\right](t)+\int_s^t\int_{\mathbb R^3}(|\Delta v|^2+a|\nabla^3u|^2)dxd\tau\\
\leq&C\left[\int_{\mathbb R^3}(|\nabla v|^2+|\Delta u|^2)dx\right](s)
\end{align*}
for any $\frac{T^*}{2}\leq s\leq t<T^*$. In particular, it holds that
$$
\sup_{\frac{T^*}{2}\leq t<T^*}\left[\int_{\mathbb R^3}(|\nabla v|^2+|\Delta u|^2)dx\right](t)+\int_{\frac{T^*}{2}}^{T^*}\int_{\mathbb R^3}(|\Delta v|^2+a|\nabla^3u|^2)dxdt<\infty.
$$
Consequently, one can apply  the local existence to extend the strong solution $(u,v)$  beyond $T^*$, which contradicts to the definition of $T^*$.

If $\omega\in L^1(\frac{T^*}{2},T^*; BMO)$, then it follows from (\ref{4.20}) that
\begin{align*}
&\left[\int_{\mathbb R^3}(|\nabla v|^2+|\Delta u|^2)dx\right](t)+\int_s^t\int_{\mathbb R^3}(|\Delta v|^2+a|\nabla^3u|^2)dxd\tau\\
\leq&C\exp\left\{\int_s^t[\omega]_{BMO}d\tau\ln\left(1+\int_s^t\|\nabla^3v\|_{L^2}d\tau\right)\right\}\left[\int_{\mathbb R^3}(|\nabla v|^2+|\Delta u|^2)dx\right](s)
\end{align*}
for any $\frac{T^*}{2}\leq s\leq t<T^*$. For any $\delta>0$, we can choose $s\in(\frac{T^*}{2}, T^*)$ such that $\int_s^t[\omega]_{BMO}d\tau<\delta$ for any $s\leq t<T^*$, and thus \begin{align}
&\left[\int_{\mathbb R^3}(|\nabla v|^2+|\Delta u|^2)dx\right](t)+\int_s^t\int_{\mathbb R^3}(|\Delta v|^2+a|\nabla^3u|^2)dxd\tau\nonumber\\
\leq&C\left[1+\left(\int_s^t\|\nabla^3v\|_{L^2}d\tau\right)^{C\delta}\right]\left[\int_{\mathbb R^3}(|\nabla v|^2+|\Delta u|^2)dx\right](s)\nonumber\\
\leq&C\left[1+\left(\int_s^t\|\nabla^3v\|_{L^2}^2d\tau\right)^{C\delta}\right]\left[\int_{\mathbb R^3}(|\nabla v|^2+|\Delta u|^2)dx\right](s)\label{4.21}
\end{align}
for any $s\leq t<T^*$. Due to (\ref{4.19}) and $|\nabla u|^2\leq|\Delta u|$, it follows from Lemma \ref{lem4.1} and Lemma \ref{lem4.2} that
\begin{align}
&\left[\int_{\mathbb R^3}(|\nabla^3u|^2+|\Delta v|^2)dx\right](t)+\int_s^t\int_{\mathbb R^3}(\frac{3a}{2}|\nabla^4u|^2+|\nabla^3v|^2)dxd\tau\nonumber\\
\leq&\exp\left\{C\min\left\{\int_s^t\|v\|_{L^{r_2}}^{q_2}d\tau,\int_s^t\|\nabla v\|_{L^\infty}d\tau\right\}\right\}\nonumber\\
&\times\exp\left\{C\min\left\{\int_s^t\|\nabla u\|_{L^{r_1}}^{q_1}d\tau,\int_s^t\|\Delta u\|_{L^\infty}d\tau\right\}\right\}\nonumber\\
&\times\left[\left(\int_{\mathbb R^3}(|\nabla^3u|^2+|\Delta v|^2)dx\right)(s)+\int_s^t\int_{\mathbb R^3}(|\nabla v|^4+|\nabla^2u|^4)dxd\tau\right]\nonumber\\
\leq&C\exp\left\{C\int_s^t\|\nabla v\|_{L^\infty}d\tau\right\}\left[\left(\int_{\mathbb R^3}(|\nabla^3u|^2+|\Delta v|^2)dx\right)(s)\right.\nonumber\\
&\left.+\int_s^t\int_{\mathbb R^3}(|\nabla v|^4+|\nabla^2u|^4)dxd\tau\right]\nonumber\\
\leq&C\exp\left\{C\int_s^t[\omega]_{BMO}d\tau\ln\left(1+\int_s^t\|\nabla^3v\|_{L^2}d\tau\right)\right\}\nonumber\\
&\times\left[\left(\int_{\mathbb R^3}(|\nabla^3u|^2+|\Delta v|^2)dx\right)(s)+\int_s^t\int_{\mathbb R^3}(|\nabla v|^4+|\nabla^2u|^4)dxd\tau\right]\nonumber\\
\leq&C\left(1+\int_s^t\|\nabla^3v\|_{L^2}d\tau\right)^{C\delta}\left[\left(\int_{\mathbb R^3}(|\nabla^3u|^2+|\Delta v|^2)dx\right)(s)\right.\nonumber\\
&\left.+\int_s^t\int_{\mathbb R^3}(|\nabla v|^4+|\nabla^2u|^4)dxd\tau\right]\nonumber\\
\leq&C\left(1+\int_s^t\|\nabla^3v\|_{L^2}^2d\tau\right)^{C\delta}\left[\left(\int_{\mathbb R^3}(|\nabla^3u|^2+|\Delta v|^2)dx\right)(s)\right.\nonumber\\
&\left.+\int_s^t\int_{\mathbb R^3}(|\nabla v|^4+|\nabla^2u|^4)dxd\tau\right]\label{4.22}
\end{align}
Set
$$
f(t)=\sup_{s\leq\tau\leq t}\int_{\mathbb R^3}(|\nabla^3u|^2+|\Delta v|^2)dx+\int_s^t\int_{\mathbb R^3}(|\nabla^4u|^2+|\nabla^3v|^2)dxd\tau.
$$
By (\ref{4.21}) and (\ref{4.22}), it holds
\begin{eqnarray}
&&\sup_{s\leq\tau\leq t}\int_{\mathbb R^3}(|\Delta u|^2+|\nabla v|^2)dx\leq C(1+f^{C\delta}(t))\left[\int_{\mathbb R^3}(|\Delta u|^2+|\nabla v|^2)dx\right](s),\label{4.23}\\
&&f(t)\leq C(1+f^{C\delta}(t))\left[f(s)+\int_s^t\int_{\mathbb R^3}(|\nabla v|^4+|\nabla^2u|^4)dxd\tau\right].\label{4.24}
\end{eqnarray}
By the Gagliado-Nirenberg-Sobolev inequality, it follows from the above two inequalities that
\begin{align*}
&\int_s^t\int_{\mathbb R^3}|\nabla v|^4dxd\tau
\leq C\int_s^t\|\nabla v\|_{L^2}\|\nabla^2v\|_{L^2}^3d\tau\\
\leq& C\int_s^t\|\nabla v\|_{L^2}\|\nabla v\|_{L^2}^{3/2}\|\nabla^3v\|_{L^2}^{3/2}d\tau\\
\leq& C\left(\int_s^t\|\nabla^3v\|_{L^2}^2d\tau\right)^{3/4}\left(\int_s^t\|\nabla v(\tau)\|_{L^2}^{10}d\tau\right)^{1/4}\\
\leq&C\left(\int_s^t\|\nabla^3v\|_{L^2}^2d\tau\right)^{3/4}(\sup_{s\leq\tau\leq t}\|\nabla v\|_{L^2}^2)^{5/4}\\
\leq&C(1+f^{3/4+C\delta}(t))\left[\int_{\mathbb R^3}(|\Delta u|^2+|\nabla v|^2)dx\right](s)
\end{align*}
and similarly
$$
\int_s^t\int_{\mathbb R^3}|\nabla^2 u|^4dxd\tau\leq C(1+f^{3/4+C\delta}(t))\left[\int_{\mathbb R^3}(|\Delta u|^2+|\nabla v|^2)dx\right](s).
$$
Combining the above two inequalities with (\ref{4.24}) yields
\begin{align*}
f(t)\leq C(1+f^{3/4+C\delta}(t))\left[f(s)+\left(\int_{\mathbb R^3}(|\Delta u|^2+|\nabla v|^2)dx\right)(s)\right]
\end{align*}
for any $s\leq t<T^*$. This, together with (\ref{4.21}), gives
$$
\sup_{s\leq t<T^*}\int_{\mathbb R^3}(|\nabla v|^2+|\Delta u|^2)dx+\int_s^{T^*}(|\Delta v|^2+|\nabla^3u|^2)dx<\infty.
$$
As a consequence, one can apply Theorem \ref{thm1} to extend $(u, v)$ to be a strong solution beyond $T^*$, which contradicts to the definition of $T^*$ again.

Now we prove the equivalency of the quantities $J_i$, $i=1,2,3,4$. Suppose that $J_1(T)$ is finite, then the statements proved in the above implies that $T$ is not the maximal existence time; as a result, $(u, v)$ can be extended to be a strong solution beyond $T$, and thus
\begin{align*}
 &u\in L^2(0, T; H_b^3(\mathbb R^3))\cap L^\infty(0, T; H_b^2(\mathbb R^3)),\\
 &v\in L^2(0, T; H^2(\mathbb R^3))\cap L^\infty(0, T; H^1(\mathbb R^3)).
\end{align*}
Due to these facts, by Lemma \ref{lem4.1}, one can easily prove that
$$
u\in L^2(\tfrac{T}{2}, T; H_b^4(\mathbb R^3)),\quad v\in L^2(\tfrac{T}{2}, T; H^3(\mathbb R^3)).
$$
Thus, one can check easily that $J_2(T), J_3(T)$ and $J_4(T)$ are all finite. Other cases can be proved in the same way.
The proof of Theorem 1 is complete. \end{proof}

Finally, it should be noted that Theorem 1 has an equivalent version:

 \begin{theorem} \label{Theorem 3}  Let $(u_0, v_0)\in H_b^{2} (\R^3; S^2)\times
H^1(\R^3,\R^3)$  be  given  initial data with $\text{div } v_0=0$.
Then, there exists a unique strong solution $(u, v):\R^3\times
[0,T^*)\to S^2\times \R^3$ of (\ref{1.5})-(\ref{1.7}) with initial values
(\ref{1.9}). Moreover, the maximal time $T^*$ can be
characterized by the condition that there are two constants
$\varepsilon_0>0$ and $R_0>0$ such that at a singular point
$ x_{i}$,
$$\limsup_{t \nearrow  T^*} \int_{B_{R}( x_{i})} |\nabla
u\left(\cdot , t\right)|^3 +|v\left(\cdot , t\right)|^3  \,dx \geq
\varepsilon_{0}$$
 for any  $R>0$ with $R\leq R_0$.
 \end{theorem}

To see this, we note that in the proof of Lemma 3.1, we have proved that
\begin{align*}
&\frac{d}{dt}\int_{\mathbb R^3}(|\nabla^3u|^2+|\Delta v|^2)dx+\int_{\mathbb R^3}(a|\nabla^4u|^2+|\nabla^3v|^2)dx\nonumber\\
\leq&C\int_{\mathbb R^3}(|\nabla u|^2+|v|^2)(|\Delta v|^2+|\nabla^3u|^2)\,dx.
\end{align*}
By a standard covering argument of $\R^3$, one can obtain
\begin{align*}
&\quad\int_{\mathbb R^3}(|\nabla u|^2+|v|^2)(|\Delta v|^2+|\nabla^3u|^2)\,dx\\
&\leq C\sum_i \int_{B_{R_0}(x_i)}(|\nabla u|^2+|v|^2)(|\Delta v|^2+|\nabla^3u|^2)\,dx
\\
&\leq C\sum_i \left [\int_{B_{R_0}(x_i)}(|\nabla u|^3+|v|^3)\right ]^{2/3}\left [\int_{B_{R_0}(x_i)} (|\Delta v|^6+|\nabla^3u|^6)\,dx\right ]^{1/3}\\
&\leq C\varepsilon_0^{2/3}\int_{\R^3}(|\nabla^4u|^2+|\nabla^3v|^2+\frac 1{R_0^2}(|\nabla^3 u|^2 +|\Delta v|^2)dx.
\end{align*}
We now can prove Theorem \ref{Theorem 3} by similar arguments as before, details are ommitted.

\section{Convergence of Ginzburg-Landau to Ericksen-Leslie}\label{sec4}

In this section, we prove that the strong solutions to the Ginzburg-Landau approximate system converge to the strong solution of the Ericksen-Leslie system and give a new blow up criterion of the strong solutions to the Ericksen-Leslie system in term of Serrin type norms of the strong solutions to the Ginzburg-Landau approximate system.

The following lemma is a characterization of precompact subset of $L^p(\mathbb R^N)$.

\begin{lemma}\label{lemA} (see Adams and Fournier \cite{Adams} Theorem 2.32) Let $1\leq p<\infty$. A bounded subset $K\subseteq L^p(\mathbb R^N)$ is precompact in $L^p(\mathbb R^N)$ if and only if for every number $\varepsilon>0$ there exists a number $\delta>0$ and a compact subset $G$ such that for every $u\in K$ and $h\in\mathbb R^N$ with $|h|<\delta$ both of the following inequalities hold:
$$
\int_{\mathbb R^N}|u(x+h)-u(x)|^pdx\leq\varepsilon^p,\qquad\int_{\mathbb R^N\setminus G}|u(x)|^pdx\leq\varepsilon^p.
$$

\end{lemma}

We need the following local type energy inequality.

\begin{lemma}\label{lem5.2}
Let $(u_\varepsilon, v_\varepsilon)$ be a strong solution to the system (\ref{1.10})--(\ref{1.12}) in $\mathbb R^3\times(0, T)$, satisfying $\frac{7}{8}\leq|u_\varepsilon|\leq\frac{9}{8}$ on $\mathbb R^3\times(0, T)$.  Then for any $\varphi\in C^\infty(\mathbb R^3)\cap W^{1,\infty}(\mathbb R^3)$, there holds
\begin{align*}
&\frac{d}{dt}\int_{\mathbb R^3}\left[|v_\varepsilon|^2+\frac{(1-|u_\varepsilon|^2)^2}{2\varepsilon^2}+2W(u_\varepsilon,\nabla u_\varepsilon)\right]\varphi^2dx\\
&+\int_{\mathbb R^3}(|\nabla v_\varepsilon|^2+|\partial_tu_\varepsilon+v_\varepsilon\nabla u_\varepsilon|^2)\varphi^2dx\\
\leq&C\int_{\mathbb R^3}[(|v_\varepsilon|^2+|\nabla u_\varepsilon|^2+|p_\varepsilon|+|\nabla v_\varepsilon|+|\nabla^2u_\varepsilon|)|v_\varepsilon||\varphi||\nabla\varphi|\\
&+(|v_\varepsilon|^2+|\nabla u_\varepsilon|^2)|\nabla\varphi|^2]dx
\end{align*}
for any $t\in(0, T)$, where $C$ is an absolute constant.
\end{lemma}

\begin{proof}
Multiplying (\ref{1.9}) by $v^i_\varepsilon\varphi^2$ and integrating over $\mathbb R^3$ yield
\begin{align}
&\frac{d}{dt}\int_{\mathbb R^3}\frac{|v_\varepsilon|^2}{2}\varphi^2dx+\int_{\mathbb R^3}|\nabla v_\varepsilon|^2\varphi^2dx\nonumber\\
=&\int_{\mathbb R^3}\left[\left(\frac{|v_\varepsilon|^2}{2}+p_\varepsilon\right)\textmd{div}(v_\varepsilon\varphi^2)-\frac{1}{2}\nabla|v_\varepsilon|^2\nabla\varphi^2\right]dx\nonumber\\
&+\int_{\mathbb R^3}W_{p_j^k}(u_\varepsilon,\nabla u_\varepsilon)\nabla u_\varepsilon^k\nabla_j(v_\varepsilon\varphi^2)dx.\label{5.12}
\end{align}

Multiplying (\ref{1.11}) by $(\partial_tu_\varepsilon^i+v_\varepsilon\nabla u_\varepsilon^i)\varphi^2$ and integrating over $\mathbb R^3$, one can get
\begin{align}
&\int_{\mathbb R^3}(\partial_tu_\varepsilon^i+v_\varepsilon\cdot\nabla u_\varepsilon^i)^2\varphi^2dx\nonumber\\
=&\int_{\mathbb R^3}\left[\nabla_\alpha(W_{p_\alpha^i}(u_\varepsilon,\nabla u_\varepsilon)-W_{u^i}(u_\varepsilon,\nabla u_\varepsilon)+\frac{1-|u_\varepsilon|^2}{\varepsilon^2}u^i_\varepsilon\right](\partial_tu_\varepsilon^i+v_\varepsilon\cdot\nabla u_\varepsilon^i)\varphi^2dx. \label{5.13}
\end{align}
It follows from integrating by parts that
\begin{align*}
&\int_{\mathbb R^3}[\nabla_\alpha(W_{p_\alpha^i}(u_\varepsilon,\nabla u_\varepsilon))-W_{u^i}(u_\varepsilon,\nabla u_\varepsilon)]\partial_tu_\varepsilon^i\varphi^2dx\\
=&-\int_{\mathbb R^3}[W_{p_\alpha^i}(u_\varepsilon,\nabla u_\varepsilon)\partial_t\nabla_\alpha u_\varepsilon^i+W_{u^i}(u_\varepsilon,\nabla u_\varepsilon)\partial_tu_\varepsilon^i]\varphi^2dx\\
&-\int_{\mathbb R^3}W_{p_\alpha^i}(u_\varepsilon,\nabla u_\varepsilon)\partial_tu_\varepsilon^i\nabla_\alpha\varphi^2dx\\
=&-\int_{\mathbb R^3}[\partial_tW(u_\varepsilon,\nabla u_\varepsilon)\varphi^2+W_{p_\alpha^i}(u_\varepsilon,\nabla u_\varepsilon)\partial_tu_\varepsilon^i\nabla_\alpha\varphi^2]dx\\
=&-\frac{d}{dt}\int_{\mathbb R^3}W(u_\varepsilon,\nabla u_\varepsilon)\varphi^2dx-\int_{\mathbb R^3}W_{p_\alpha^i}(u_\varepsilon,\nabla u_\varepsilon)\partial_tu_\varepsilon^i\nabla_\alpha\varphi^2dx
\end{align*}
and
\begin{align*}
&\int_{\mathbb R^3}[\nabla_\alpha(W_{p_\alpha^i}(u_\varepsilon,\nabla u_\varepsilon))-W_{u^i}(u_\varepsilon,\nabla u_\varepsilon)]v\nabla u_\varepsilon^i\varphi^2dx\\
=&-\int_{\mathbb R^3}[W_{p_\alpha^i}(u_\varepsilon,\nabla u_\varepsilon)v\nabla\nabla_\alpha u_\varepsilon^i+W_{u^i}(u_\varepsilon, \nabla u_\varepsilon)v\nabla u_\varepsilon^i]\varphi^2dx\\
&-\int_{\mathbb R^3}W_{p_\alpha^i}(u_\varepsilon, \nabla u_\varepsilon)\nabla u_\varepsilon^i\nabla_\alpha(v_\varepsilon\varphi^2)dx\\
=&-\int_{\mathbb R^3}[v\nabla W(u_\varepsilon, \nabla u_\varepsilon)\varphi^2+W_{p_\alpha^i}(u_\varepsilon, \nabla u_\varepsilon)\nabla u_\varepsilon^i\nabla_\alpha(v_\varepsilon\varphi^2)]dx\\
=&\int_{\mathbb R^3}[W(u_\varepsilon, \nabla u_\varepsilon)\textmd{div}(v_\varepsilon\varphi^2)-W_{p_{\alpha}^i}(u_\varepsilon, \nabla u_\varepsilon)\nabla u_\varepsilon^i\nabla_\alpha(v_\varepsilon\varphi^2)]dx.
\end{align*}
Moreover, direct calculations give
\begin{align*}
&\int_{\mathbb R^3}\frac{1-|u_\varepsilon|^2}{\varepsilon^2}u^i_\varepsilon(\partial_tu_\varepsilon^i+v_\varepsilon\cdot\nabla u_\varepsilon^i)\varphi^2dx=-\int_{\mathbb R^3}(\partial_t+v_\varepsilon\cdot\nabla)\left[\frac{(1-|u_\varepsilon|^2)^2}{4\varepsilon^2}\right]\varphi^2dx\\
=&-\frac{d}{dt}\int_{\mathbb R^3}\frac{(1-|u_\varepsilon|^2)^2}{4\varepsilon^2}\varphi^2dx+\int_{\mathbb R^3}\frac{(1-|u_\varepsilon|^2)^2}{4\varepsilon^2}\textmd{div}(v_\varepsilon\varphi^2)dx.
\end{align*}
Substituting the above three equalities into (\ref{5.13}) gives
\begin{align}
&\frac{d}{dt}\int_{\mathbb R^3}\left[\frac{(1-|u_\varepsilon|^2)^2}{4\varepsilon^2}+W(u_\varepsilon, \nabla u_\varepsilon)\right]\varphi^2dx+\int_{\mathbb R^3}|\partial_tu_\varepsilon+v_\varepsilon\cdot\nabla u_\varepsilon|^2\varphi^2dx\nonumber\\
=&\int_{\mathbb R^3}\left[\left(\frac{(1-|u_\varepsilon|^2)^2}{4\varepsilon^2}+W(u_\varepsilon, \nabla u_\varepsilon)\right)\textmd{div}(v_\varepsilon\varphi^2)-W_{p_\alpha^i}(u_\varepsilon, \nabla u_\varepsilon)\partial_tu_\varepsilon^i\nabla_\alpha\varphi^2\right]dx\nonumber\\
&-\int_{\mathbb R^3}W_{p_\alpha^i}(u_\varepsilon, \nabla u_\varepsilon)\nabla u_\varepsilon^i\nabla_\alpha(v_\varepsilon\varphi^2)dx. \label{5.14}
\end{align}

Combining (\ref{5.12}) with (\ref{5.14}), we obtain
\begin{align*}
&\frac{d}{dt}\int_{\mathbb R^3}\left[\frac{|v_\varepsilon|^2}{2}+\frac{(1-|u_\varepsilon|^2)^2}{4\varepsilon^2}+W(u_\varepsilon, \nabla u_\varepsilon)\right]\varphi^2dx\\
&+\int_{\mathbb R^3}(|\nabla v_\varepsilon|^2+|\partial_tu_\varepsilon+v_\varepsilon\cdot\nabla u_\varepsilon|^2)\varphi^2dx\nonumber\\
=&\int_{\mathbb R^3}\left[\left(\frac{|v_\varepsilon|^2}{2}+p_\varepsilon+\frac{(1-|u_\varepsilon|^2)^2}{4\varepsilon^2}+W(u_\varepsilon, \nabla u_\varepsilon)\right)\textmd{div}(v_\varepsilon\varphi^2)\right.\\
&\left.-\frac{1}{2}\nabla|v_\varepsilon|^2\nabla\varphi^2-W_{p_\alpha^i}(u_\varepsilon, \nabla u_\varepsilon)\partial_tu_\varepsilon^i\nabla_\alpha\varphi^2\right]dx\nonumber\\
\leq&\eta\int_{\mathbb R^3}|\partial_tu_\varepsilon+v_\varepsilon\cdot\nabla u_\varepsilon|^2\varphi^2dx+C\int_{\mathbb R^3}[(|v_\varepsilon|^2+|\nabla u_\varepsilon|^2\\
&+|p_\varepsilon|+|\nabla v_\varepsilon|)|v_\varepsilon||\varphi||\nabla\varphi|+|\nabla u_\varepsilon|^2|\nabla\varphi|^2]dx\\
&+C\int_{\mathbb R^3}\frac{(1-|u_\varepsilon|^2)^2}{\varepsilon^2}|v_\varepsilon||\varphi||\nabla\varphi|dx.
\end{align*}
This, together with the facts that $\frac{7}{8}\leq|u_\varepsilon|\leq\frac{9}{8}$ and
$$
\left|\frac{1-|u_\varepsilon|^2}{\varepsilon^2}\right|\leq C(|\partial_tu_\varepsilon+v_\varepsilon\cdot\nabla u_\varepsilon|+|\nabla^2u_\varepsilon|+|\nabla u_\varepsilon|^2),
$$
gives
\begin{align*}
&\frac{d}{dt}\int_{\mathbb R^3}\left[\frac{|v_\varepsilon|^2}{2}+\frac{(1-|u_\varepsilon|^2)^2}{4\varepsilon^2}+W(u_\varepsilon, \nabla u_\varepsilon)\right]\varphi^2dx\\
&+\int_{\mathbb R^3}\left(|\nabla v_\varepsilon|^2+\frac{3}{4}|\partial_tu_\varepsilon+v_\varepsilon\cdot\nabla u_\varepsilon|^2\right)\varphi^2dx\nonumber\\
\leq&C\int_{\mathbb R^3}[(|v_\varepsilon|^2+|\nabla u_\varepsilon|^2+|p_\varepsilon|+|\nabla v_\varepsilon|)|v_\varepsilon||\varphi||\nabla\varphi|+|\nabla u_\varepsilon|^2|\nabla\varphi|^2]dx\\
&+C\int_{\mathbb R^3}(|\partial_tu_\varepsilon+v_\varepsilon\nabla u_\varepsilon|+|\nabla^2u_\varepsilon|+|\nabla u_\varepsilon|^2)|v_\varepsilon||\varphi||\nabla\varphi|dx\\
\leq&\eta\int_{\mathbb R^3}|\partial_tu_\varepsilon+v_\varepsilon\nabla u_\varepsilon|^2\varphi^2dx+C\int_{\mathbb R^3}[(|v_\varepsilon|^2+|\nabla u_\varepsilon|^2\\
&+|p_\varepsilon|+|\nabla v_\varepsilon|+|\nabla^2u_\varepsilon|)|v_\varepsilon||\varphi||\nabla\varphi|+(|v_\varepsilon|^2+|\nabla u_\varepsilon|^2)|\nabla\varphi|^2]dx,
\end{align*}
which implies the conclusion. This completes a proof.
\end{proof}

The following lemma will be used in the proof of strong convergence and uniform estimates.

\begin{lemma}
\label{lem4.3}
Let $(u_\varepsilon, v_\varepsilon)$ and $(u, v)$ be strong solutions to the systems (\ref{1.10})--(\ref{1.12}) and (\ref{1.5})--(\ref{1.7}) in $\mathbb R^3\times(0, T)$ with the same initial data $(u_0, v_0)$, respectively. Suppose that
$$
(\nabla u_\varepsilon, v_\varepsilon)\rightarrow(\nabla u, v)\quad\text{in }L^2(0,T; H^1(\mathbb R^3))
$$
and
$$
\varlimsup_{\varepsilon\rightarrow0}\sup_{0\leq t\leq T}\|(\nabla u_\varepsilon, v_\varepsilon)\|_{H^1(\mathbb R^3)}<\infty.
$$
Let $K>0$ be a constant such that
$$
\sup_{0\leq t\leq T}\|(\nabla u, v)\|_{H^1(\mathbb R^3)}^2+\int_0^T(\|\nabla^2u\|_2^4+\|\nabla v\|_2^4)dt\leq K.
$$
Then there are two positive constants $\varepsilon_0$ and $S_0$, with $S_0$ depending only on the initial data $(u_0, v_0)$ and $K$, such that
\begin{align*}
&\sup_{0<\varepsilon\leq\varepsilon_0}\sup_{0\leq t\leq T}\left(\|(\nabla u_\varepsilon, v_\varepsilon)\|_{H^1(\mathbb R^3)}^2+\left\|\partial_tu_\varepsilon\right\|_{L^2(\mathbb R^3)}^2\right)\leq S_0.
\end{align*}
\end{lemma}

\begin{proof} Set $M_1=\|\nabla u_0\|_{H^1(\mathbb R^3)}+\|v_0\|_{H^1(\mathbb R^3)}.$ Using  Lemma \ref{lem5.1} and the Gagliado-Nirenberg-Sobolev inequality, we have
\begin{align*}
&\frac{d}{dt}(\|\Delta u_\varepsilon\|_{2}^2+\|\partial_tu_\varepsilon\|_2^2+\|\nabla v_\varepsilon\|_2^2)\\
&+a(\|\nabla^3u_\varepsilon\|_2^2+\|\nabla\partial_tu_\varepsilon\|_2^2)+(\|\Delta v_\varepsilon\|_2^2+\|\partial_tv_\varepsilon\|_2^2)\\
\leq&C\int_{\mathbb R^3}(|\nabla u_\varepsilon|^2+|v_\varepsilon|^2)(|\nabla^2u_\varepsilon|^2+|\partial_tu_\varepsilon|^2+|\nabla v_\varepsilon|^2)dx\\
\leq&C(\|\nabla^2u_\varepsilon\|_2^2+\|\nabla v_\varepsilon\|_2^2)(\|\nabla^3u_\varepsilon\|_2+\|\nabla\partial_tu_\varepsilon\|_2+\|\nabla^2v_\varepsilon\|_2)\\
&\times(\|\Delta u_\varepsilon\|_2+\|\partial_tu_\varepsilon\|_2+\|\nabla v_\varepsilon\|_2)\\
\leq&\eta(\|\nabla^3u_\varepsilon\|_2^2+\|\nabla\partial_tu_\varepsilon\|_2^2+\|\nabla^2v_\varepsilon\|_2^2)+C(\|\nabla^2u_\varepsilon\|_2^4+\|\nabla v_\varepsilon\|_2^4)\\
&\times(\|\Delta u_\varepsilon\|_2^2+\|\partial_tu_\varepsilon\|_2^2+\|\nabla v_\varepsilon\|_2^2).
\end{align*}
Thus
\begin{align}
&\frac{d}{dt}(\|\Delta u_\varepsilon\|_{2}^2+\|\partial_tu_\varepsilon\|_2^2+\|\nabla v_\varepsilon\|_2^2)\nonumber\\
\leq&C(\|\nabla^2u_\varepsilon\|_2^4+\|\nabla v_\varepsilon\|_2^4)(\|\Delta u_\varepsilon\|_2^2+\|\partial_tu_\varepsilon\|_2^2+\|\nabla v_\varepsilon\|_2^2).\label{6.6}
\end{align}

Using equation (\ref{1.12}) and $|u_0|=1$, one has
$$
\|\partial_tu_\varepsilon(0)\|_2^2\leq C(\|\nabla^2u_0\|_2^2+\|\nabla u_0\|_4^4+\|v_0\|_4^4)\leq C(M_1^4+1).
$$
Due to the assumptions in the lemma, there is a constant $\varepsilon_0$, such that for any $\varepsilon\in(0, \varepsilon_0]$, it holds that
\begin{align*}
&\int_0^{T}(\|\nabla^2u_\varepsilon\|_2^4+\|\nabla v_\varepsilon\|_2^4)dt\\
\leq&8\int_0^{T}(\|\nabla^2u\|_2^4+\|\nabla v\|_2^4)dt+8\int_0^{T}(\|\nabla^2(u_\varepsilon-u)\|_2^4+\|\nabla (v_\varepsilon-v)\|_2^4)dt\\
\leq&8K+C\sup_{0\leq t\leq T}(\|(\nabla^2u_\varepsilon,\nabla v_\varepsilon)\|_2^2+\|(\nabla^2u, \nabla v)\|_2^2)\\
&\times\int_0^{T}(\|\nabla^2(u_\varepsilon-u)\|_2^2+\|\nabla (v_\varepsilon-v)\|_2^2)dt\\
\leq&8K+C(K+2\varlimsup_{\varepsilon\rightarrow0}\|(\nabla u_\varepsilon, v_\varepsilon)\|_{H^1(\mathbb R^3)}^2)\int_0^{T}(\|\nabla^2(u_\varepsilon-u)\|_2^2+\|\nabla (v_\varepsilon-v)\|_2^2)dt\\
\leq&8K+1.
\end{align*}
It follows from these two inequalities and (\ref{6.6}) that
\begin{align}
&\sup_{0\leq t\leq T}(\|\Delta u_\varepsilon\|_{2}^2+\|\partial_tu_\varepsilon\|_2^2+\|\nabla v_\varepsilon\|_2^2)\nonumber\\
\leq&e^{C\int_0^{T}(\|\nabla^2u_\varepsilon\|_2^4+\|\nabla v_\varepsilon\|_2^4)dt}(\|\Delta u_0\|_{2}^2+\|\partial_tu_\varepsilon(0)\|_2^2+\|\nabla v_0\|_2^2)\nonumber\\
\leq&C(M_1^4+1)e^{CK+C}:=M_2^2\label{6.7}
\end{align}
for any $\varepsilon\in(0, \varepsilon_0].$
By Lemma \ref{LLEM2.2},
$$
\sup_{0\leq t\leq T}\|v_\varepsilon\|_2^2+\|\nabla u_\varepsilon\|_2^2\leq CM_1^2.
$$
Combining this with (\ref{6.7}), we have
\begin{align*}
\sup_{0<\varepsilon\leq\varepsilon_0}\sup_{0\leq t\leq T}\left(\|(\nabla u_\varepsilon, v_\varepsilon)\|_{H^1(\mathbb R^3)}^2+\left\|\partial_tu_\varepsilon\right\|_{L^2(\mathbb R^3)}^2\right)
\leq C(M_1^2+M_2^2):=S_0.
\end{align*}
This completes the proof.
\end{proof}

The following lemma will be used to prove the new blow up criterion.

\begin{lemma}\label{lem4.4}
Let $(u_\varepsilon, v_\varepsilon)$ be a strong solution in $\mathbb R^3\times(0, T_\varepsilon)$ to the system (\ref{1.10})--(\ref{1.12}) with (\ref{1.9}). Suppose that
$$
\|(\nabla u_\varepsilon, v_\varepsilon)\|_{L^q(0, T_\varepsilon; L^r(\mathbb R^3))}\leq L
$$
for some positive constant $L$ and $\frac{2}{q}+\frac{3}{r}=1$ with $q\in[2,\infty)$ and $r\in(3,\infty]$. Then there is a constant $N$ depending only on $L$ and the initial data such that
$$
\sup_{0\leq t\leq T_\varepsilon}\|(\nabla u_\varepsilon, v_\varepsilon)\|_{H^1}^2+\int_0^{T_\varepsilon}(\|(\nabla^3u_\varepsilon, \nabla^2v_\varepsilon)\|_{L^2}^2+\|(\nabla\partial_tu_\varepsilon, \partial_tv_\varepsilon)\|_{L^2}^2)dt\leq N
$$
for small $\varepsilon$.
\end{lemma}

\begin{proof}
Let $T_\varepsilon^1\in(0, T_\varepsilon]$ be the maximal time such that $\frac{3}{4}\leq|u_\varepsilon|\leq\frac{5}{4}$ on $\mathbb R^3\times[0, T_\varepsilon^1)$. By Lemma \ref{lem5.1},
\begin{align*}
&\frac{d}{dt}\int_{\mathbb R^3}(|\Delta u_\varepsilon|^2+|\partial_t u_\varepsilon|^2+|\nabla v_\varepsilon|^2)dx\\
&+\int_{\mathbb R^3}\left(a|\nabla^3u_\varepsilon|^2+a|\nabla\partial_tu_\varepsilon|^2+|\Delta v_\varepsilon|^2+|\partial_tv_\varepsilon|^2\right)dx\nonumber\\
\leq&C\int_{\mathbb R^3}(|\nabla u_\varepsilon|^2+|v_\varepsilon|^2)(|\nabla^2u_\varepsilon|^2+|\partial_tu_\varepsilon|^2+|\nabla v_\varepsilon|^2)dx
\end{align*}
for any $t\in(0, T_\varepsilon^1)$. By the Sobolev embedding inequality and the H\"older inequality, it holds that
\begin{align*}
\int_{\mathbb R^3}|w|^2|f|^2dx\leq \|w\|_{L^r}^2\|f\|_{L^2}^{\frac{2r-6}{r}}\|f\|_{L^6}^{\frac{6}{r}}\leq\eta\|\nabla f\|_{L^2}^2+C\|w\|_{L^r}^q\|f\|_{L^2}^2.
\end{align*}
Combining the above two inequalities shows that
\begin{align*}
&\frac{d}{dt}\int_{\mathbb R^3}(|\Delta u_\varepsilon|^2+|\partial_t u_\varepsilon|^2+|\nabla v_\varepsilon|^2)dx\\
&+\frac{1}{2}\int_{\mathbb R^3}\left(a|\nabla^3u_\varepsilon|^2+a|\nabla\partial_tu_\varepsilon|^2+|\Delta v_\varepsilon|^2+|\partial_tv_\varepsilon|^2\right)dx\nonumber\\
\leq&C\|(\nabla u_\varepsilon, v_\varepsilon)\|_{L^q(0, T_\varepsilon; L^r(\mathbb R^3))}(\|\nabla^2 u_\varepsilon\|_{L^2(\mathbb R^3)}^2+\|\partial_t u_\varepsilon\|_{L^2(\mathbb R^3)}^2+\|\nabla v_\varepsilon\|_{L^2(\mathbb R^3)}^2).
\end{align*}
Equation (\ref{1.12}) implies that
$$
\|\partial_tu_\varepsilon(0)\|_{L^2(\mathbb R^3)}^2\leq C(\|\nabla^2u_0\|_{L^2(\mathbb R^3)}^2+\|v_0\|_{L^4(\mathbb R^3)}^4+\|\nabla u_0\|_{L^4(\mathbb R^3)}^4)\leq C.
$$
It follows from above two inequalities that
\begin{align}
&\sup_{0\leq t\leq T_\varepsilon^1}(\|\nabla^2u_\varepsilon\|_{L^2(\mathbb R^3)}^2+\|\partial_tu_\varepsilon\|_{L^2(\mathbb R^3)}^2+\|\nabla v_\varepsilon\|_{L^2(\mathbb R^3)}^2)\nonumber\\
&+\int_0^{T_\varepsilon^1}(\|\nabla^3u_\varepsilon\|_{L^2(\mathbb R^3)}^2+\|\nabla\partial_tu_\varepsilon\|_{L^2(\mathbb R^3)}^2+\|\nabla^2v_\varepsilon\|_{L^2(\mathbb R^3)}^2+\|\partial_tv_\varepsilon\|_{L^2(\mathbb R^3)}^2)dt\nonumber\\
\leq&Ce^{\int_0^{T_\varepsilon^1}(\|\nabla u_\varepsilon\|_{L^r}^q+\|v_\varepsilon\|_{L^r}^q)dt}\leq Ce^{L^q}.\label{7.2}
\end{align}
Apply Lemma \ref{LLEM2.2} to obtain
\begin{equation*}
\sup_{0\leq t\leq T_\varepsilon^1}(\|(\nabla u_\varepsilon, \nabla v_\varepsilon)\|_{L^2(\mathbb R^3)}^2+\varepsilon^{-2}\|1-|u_\varepsilon|^2\|_{L^2(\mathbb R^3)}^2)\leq C\|(\nabla u_0, v_0)\|_{L^2(\mathbb R^3)}^2, \label{7.3}
\end{equation*}
which, together with (\ref{7.2}), gives
\begin{equation*}
\sup_{0\leq t\leq T_\varepsilon^1}\|(\nabla^2u_\varepsilon, v_\varepsilon)\|_{H^1}^2+\int_0^{T_\varepsilon^1}(\|(\nabla^3u_\varepsilon, \nabla\partial_tu_\varepsilon)\|_{L^2}^2+\|(\nabla^2v_\varepsilon, \partial_tv_\varepsilon)\|_{L^2}^2)dt\leq N\label{7.4}
\end{equation*}
for some constant $N$ depending only on $L$ and the initial data. By the aid of the above two inequalities, one can use a similar argument as in the proof of \ref{prop} to conclude by using the Gagliado-Nirenberg-Sobolev inequality that
$$
\frac{7}{8}\leq|u_\varepsilon|\leq\frac{9}{8}\quad\text{ on }\mathbb R^3\times[0, T_\varepsilon^1)
$$
for small $\varepsilon$. Recalling the definition of $T_\varepsilon^1$, the above inequality implies $T_\varepsilon^1=T_\varepsilon$, and thus the conclusion holds true.
\end{proof}

Now, let us give the proof of Theorem \ref{thm2}.

\begin {proof} [Proof of Theorem \ref{thm2}]

We first prove the strong convergence and the uniform estimates, which are given in three steps as follows.

Given arbitrary $T\in(0, T^*)$, set
$$
K=\sup_{0\leq t\leq T}\|(\nabla u, v)\|_{H^1(\mathbb R^3)}^2+\int_0^T(\|\nabla^2u\|_2^4+\|\nabla v\|_2^4)dt.
$$
Let $S_0$ be the constant stated in Lemma \ref{lem4.3} and put
$$
M=\|(\nabla u_0, v_0)\|_{H^1(\mathbb R^3)}^2+S_0.
$$

\textbf{Step 1.} In this step, we prove that the strong convergence and estimate hold true up to a time $T_M$.
By Proposition \ref{prop} and Lemma \ref{LLEM2.2}, $(u_\varepsilon, v_\varepsilon)$ can be defined on $\mathbb R^3\times[0, T_M]$ such that
$\frac{7}{8}\leq|u_\varepsilon|\leq\frac{9}{8}$ on $\mathbb R^3\times[0, T_M]$ and
\begin{equation}
\sup_{0\leq t\leq T_M}\|(\nabla u_\varepsilon, v_\varepsilon)\|_{H^1}^2+\int_0^{T_M}\|(\nabla^3u_\varepsilon, \nabla^2v_\varepsilon, \partial_t\nabla u_\varepsilon, \partial_tv_\varepsilon)\|_{L^2}^2dt\leq C(M)\label{6.1}
\end{equation}
for small $\varepsilon$.

Using the same argument as the proof of Theorem \ref{thm1}, we can prove that
\begin{align}
&u_\varepsilon\rightarrow u\quad\mbox{ in }\quad L^2(0, T_M; H^2(B_R(0))),\label{6.2}\\
&v_\varepsilon\rightarrow v\quad\mbox{ in }\quad L^2(0, T_M; H^1(B_R(0)))\label{6.3}
\end{align}
for any $R\in(0, \infty)$. In fact, to prove these convergence, by the aid of the uniqueness of the strong solutions to system (\ref{1.5})--(\ref{1.7}), it suffices to show that any sequence $\{(u_{\varepsilon_i}, v_{\varepsilon_i})\}_{i=1}^\infty$ has convergent subsequence. While such sequentially convergence has already been justified in the proof of Theorem \ref{thm1}.

The aim is to show that
\begin{equation}
\nabla u_\varepsilon\rightarrow\nabla u\quad\mbox{and}\quad v_\varepsilon\rightarrow v\quad\mbox{ in } L^2(0,T_M; H^1(\mathbb R^3)).\label{6.4}
\end{equation}
By the aid of (\ref{6.1})--(\ref{6.3}), using Lemma \ref{lemA} and the Gagliado-Nirenberg-Sobolev inequality, one needs to show that for any $\eta>0$, there is $R>0$, such that
\begin{equation}\label{6.5}
\int_0^{T_M}\int_{\mathbb R^3\setminus B_R(0)}(|\nabla u_\varepsilon|^2+|v_\varepsilon|^2)dxdt\leq\eta.
\end{equation}
Take function $\varphi_0\in C^\infty(\mathbb R)\cap W^{1,\infty}(\mathbb R)$, such that $\varphi_0\equiv0$ on $(-\infty, 1)$, $\varphi_0\equiv1$ on $(2,\infty)$ and $|\varphi'|\leq 2$ on $\mathbb R$. For $R\geq 1$, set $\varphi_R(x)=\varphi_0(\frac{|x|}{R})$, then $\varphi_R(x)\equiv0$ on $B_R(0)$, $\varphi_R\equiv1$ on $\mathbb R^3\setminus B_{2R}(0)$ and $|\nabla\varphi_R|\leq\frac{2}{R}$ on $\mathbb R^3$. For $\varphi=\varphi_R$ in Lemma \ref{lem5.2}, it holds that
\begin{align*}
&\sup_{0\leq t\leq T_M}\int_{\mathbb R^3\setminus B_{2R}(0)}(|v_\varepsilon|^2+|\nabla u_\varepsilon|^2)dx\\
\leq&\int_{\mathbb R^3\setminus B_R(0)}(|v_0|^2+|\nabla u_0|^2)dx+\frac{C}{R}\int_0^{T_M}\int_{\mathbb R^3}(|v_\varepsilon|^4+|\nabla u_\varepsilon|^4\\
&+|p_\varepsilon|^2+|\nabla v_\varepsilon|^2+|\nabla^2u_\varepsilon|^2+|v_\varepsilon|^2+|\nabla u_\varepsilon|^2)dxdt.
\end{align*}
Applying elliptic estimates for the Stokes equations, it follows from equation (\ref{1.10}) that
$$
\int_0^{T_M}\int_{\mathbb R^3}|p_\varepsilon|^2dxdt\leq C\int_0^{T_M}\int_{\mathbb R^3}(|v_\varepsilon|^4+|\nabla u_\varepsilon|^4)dxdt.
$$
Combining the above two inequalities, using the Gagliado-Nirenber-Sobolev inequality and the absolute continuity of integrals, one obtains from (\ref{6.1}) that
\begin{align*}
&\sup_{0\leq t\leq T_M}\int_{\mathbb R^3\setminus B_{2R}(0)}(|v_\varepsilon|^2+|\nabla u_\varepsilon|^2)dx\\
\leq&\int_{\mathbb R^3\setminus B_R(0)}(|v_0|^2+|\nabla u_0|^2)dx+\frac{C}{R}\int_0^{T_M}(\|v_\varepsilon\|_{H^1}^4+\|\nabla u_\varepsilon\|_{H^1}^4+1)dt\\
\leq&\frac{\eta}{2}+\frac{C}{R}T_MC(M)\leq\eta
\end{align*}
for large $R$, which shows (\ref{6.5}) and thus (\ref{6.4}).

Next we prove
\begin{equation}
(\nabla u_\varepsilon, v_\varepsilon)\rightarrow(\nabla u, v)\quad\mbox{ in }L^\infty(0, T_M; L^2(\mathbb R^3)). \label{6.11}
\end{equation}
Due to (\ref{6.1}), it suffices to show that each sequence $(u_{\varepsilon_i}, v_{\varepsilon_i})$ has an convergent subsequence in $L^\infty(0, T; L^2(\mathbb R^3))$. Let $(u_{\varepsilon_i}, v_{\varepsilon_i})$ be a sequence. By (\ref{6.1}), there is a subsequence, still denoted by $(u_{\varepsilon_i}, v_{\varepsilon_i})$, such that
\begin{equation}
\lim_{i\rightarrow\infty}\|(\nabla u_\varepsilon(\cdot, t)-\nabla u(\cdot, t), v_\varepsilon(\cdot, t)-v(\cdot, t)\|_{H^1(\mathbb R^3)}=0,\quad \text{ for a.e. }t\in[0, T_M]. \label{6.12}
\end{equation}
By (\ref{6.1}), it holds that
\begin{align*}
&\|\nabla u_\varepsilon(\cdot, t)-\nabla u_\varepsilon(\cdot, s)\|_{L^2(\mathbb R^3)}=\left\|\int_s^t\partial_t\nabla u_\varepsilon(\tau)d\tau\right\|_{L^2(\mathbb R^3)}\\
\leq&\int_s^t\left\|\partial_t\nabla u_\varepsilon(\tau)\right\|_{L^2(\mathbb R^3)}d\tau\leq(t-s)^{1/2}\int_s^t\left\|\partial_t\nabla u_\varepsilon(\tau)\right\|_{L^2(\mathbb R^3)}^2d\tau\\
\leq& C(M)(t-s)^{1/2},
\end{align*}
and similarly
$$
\|v_\varepsilon(\cdot, t)-v_\varepsilon(\cdot, s)\|_{L^2(\mathbb R^3)}\leq C(M)(t-s)^{1/2}.
$$
By the aid of the two inequalities above and (\ref{6.12}), one can prove easily (\ref{6.11}) by a density argument. This completes the proof of Step 1.

\textbf{Step 2.} In this step, we prove that if the strong convergence and uniform estimate hold true up to time $T_1$ with $T_1<T$, then they hold true up to time $T_2:=\min\{T, T_1+T_M\}$. Suppose that $\frac{7}{8}\leq|u_\varepsilon|\leq\frac{9}{8}$ on $\mathbb R^3\times[0, T_1]$,
\begin{equation*}
\sup_{0\leq t\leq T_1}\|(\nabla u_\varepsilon, v_\varepsilon)\|_{H^1}^2+\int_0^{T_1}\|(\nabla^3u_\varepsilon, \nabla^2v_\varepsilon, \partial_t\nabla u_\varepsilon, \partial_tv_\varepsilon)\|_{L^2}^2dt\leq C(M)
\end{equation*}
and
$$
(\nabla u_\varepsilon, v_\varepsilon)\rightarrow(\nabla u, v)\quad\text{in }L^\infty(0, T_1; L^2(\mathbb R^3))\cap L^2(0, T_1; H^1(\mathbb R^3))
$$
for some $T_1<T$.
Due to the above two inequality, we apply Lemma \ref{lem4.3} to conclude that
$$
\sup_{0\leq t\leq T_1}\left(\|(\nabla u_\varepsilon, v_\varepsilon)\|_{H^1(\mathbb R^3)}^2+\left\|\partial_tu_\varepsilon\right\|_{L^2(\mathbb R^3)}^2\right)\leq S_0,
$$
which, using equation (\ref{1.12}), gives
$$
\sup_{0\leq t\leq T_1}\left(\|(\nabla u_\varepsilon, v_\varepsilon)\|_{H^1(\mathbb R^3)}^2+\left\|\mathcal Q_\varepsilon(u_\varepsilon, v_\varepsilon)\right\|_{L^2(\mathbb R^3)}^2\right)\leq S_0\leq M
$$
for small $\varepsilon$.
Recall that $\frac{7}{8}\leq|u_\varepsilon|\leq\frac{9}{8}$ on $\mathbb R^3\times[0, T_1]$, starting from time $T_1$ and taking $(u_\varepsilon(T_1), v_\varepsilon(T_1))$ as the initial data, we can apply Proposition \ref{prop} again to extend $(u_\varepsilon, v_\varepsilon)$ to time $t_2:=\min\{T_1+T_M, T\}$, such that
$\frac{7}{8}\leq|u_\varepsilon|\leq\frac{9}{8}$ on $\mathbb R^3\times[0, T_2]$ and
\begin{equation*}
\sup_{0\leq t\leq T_2}\|(\nabla u_\varepsilon, v_\varepsilon)\|_{H^1}^2+\int_0^{T_2}\|(\nabla^3u_\varepsilon, \nabla^2v_\varepsilon, \partial_t\nabla u_\varepsilon, \partial_tv_\varepsilon)\|_{L^2}^2dt\leq C(M)
\end{equation*}
for small $\varepsilon$. Using the same argument as Step 1, we can prove the strong convergence up to time $T_2$, that is
\begin{equation*}
\nabla u_\varepsilon\rightarrow\nabla u\quad\mbox{and}\quad v_\varepsilon\rightarrow v\quad\mbox{ in } L^2(0, T_2; H^1(\mathbb R^3))\cap L^\infty(0, T_2; L^2(\mathbb R^3)).
\end{equation*}
This completes the proof of Step 2.

\textbf{Step 3.} Combining Step 1 and Step 2, one can easily prove that the strong convergence and uniform estimate hold true for any time $T<T^*$.

Now, we turn to the proof of the characterization of the maximal existence time. Suppose that $T^*<\infty$ is the maximal existence time of the strong solution $(u, v)$ to the system (\ref{1.5})--(\ref{1.9}). We prove that
\begin{equation}
\varlimsup_{\varepsilon\rightarrow0}\|(\nabla u_\varepsilon, v_\varepsilon)\|_{L^q(0, T^*; L^r(\mathbb R^3))}=\infty\label{7.5}
\end{equation}
for any $(q, r)\in \mathcal O$. Suppose, by contradiction, that the above is not true. Then there is some $(q, r)\in\mathcal O$ and a sequence $\varepsilon_i\rightarrow0$, such that
$$
\|(\nabla u_{\varepsilon_i}, v_{\varepsilon_i})\|_{L^q(0, T^*; L^r(\mathbb R^3))}\leq L
$$
for a positive number $L$. By Lemma \ref{lem4.4}, there is a positive constant $N$ depending only on $L$ and the initial data, such that
\begin{equation*}
\sup_{0\leq t\leq T^*}\|(\nabla^2u_{\varepsilon_i}, v_{\varepsilon_i})\|_{H^1}^2+\int_0^{T^*}(\|(\nabla^3u_{\varepsilon_i}, \nabla\partial_tu_{\varepsilon_i})\|_{L^2}^2+\|(\nabla^2v_{\varepsilon_i}, \partial_tv_{\varepsilon_i})\|_{L^2}^2)dt\leq N.
\end{equation*}
Due to this estimates, using the same argument to the proof of Theorem \ref{thm1} in Section \ref{sec2}, a subsequence of $(u_{\varepsilon_i}, v_{\varepsilon_i})$ converges to $(u, v)$ and
\begin{equation*}
\sup_{0\leq t\leq T^*}\|(\nabla^2u, v)\|_{H^1}^2+\int_0^{T^*}(\|(\nabla^3u, \nabla\partial_tu)\|_{L^2}^2+\|(\nabla^2v, \partial_tv)\|_{L^2}^2)dt\leq N.
\end{equation*}
As a result, by Theorem \ref{thm1}, we can extend the strong solution $(u, v)$ beyond $T^*$, contradicting to the definition of $T^*$. This contradiction implies that (\ref{7.5}) holds true.

Now we prove that (\ref{7.5}) implies that $T^*$ is the maximal existence time. Suppose, by contradiction, that $T^*$ is not the maximal existence time. By what we proved in the strong convergence and uniform estimates, we have
$$
\varlimsup_{\varepsilon\rightarrow0}\|(\nabla u_\varepsilon, v_\varepsilon)\|_{L^\infty(0, T^*; H^1(\mathbb R^3))\cap W^{2,1}_2(\mathbb R^3\times[0, T^*])}\leq M,
$$
for a positive constant $M$. It follows from the Gagliado-Nirenberg-Sobolev inequality and the above inequality that
\begin{align*}
&\|(\nabla u_\varepsilon, v_\varepsilon)\|_{L^q(0, T^*; L^r(\mathbb R^3))}=\left(\int_0^{T^*}\|(\nabla u_\varepsilon, v_\varepsilon)\|_{L^r(\mathbb R^3)}^q\right)^{1/q}\\
\leq&C\left(\int_0^{T^*}\|(\nabla u_\varepsilon, v_\varepsilon)\|_{H^1(\mathbb R^3)}^q\right)^{1/q}\leq C\sup_{0\leq t\leq T^*}\|(\nabla u_\varepsilon, v_\varepsilon)\|_{H^1(\mathbb R^3)}{T^*}^{1/q}\leq CM{T^*}^{1/q}
\end{align*}
for $r\in(3, q]$, and
\begin{align*}
&\|(\nabla u_\varepsilon, v_\varepsilon)\|_{L^q(0, T^*; L^r(\mathbb R^3))}=\left(\int_0^{T^*}\|(\nabla u_\varepsilon, v_\varepsilon)\|_{L^r(\mathbb R^3)}^q\right)^{1/q}\\
\leq&C\left(\int_0^{T^*}\|(\nabla^2 u_\varepsilon, \nabla v_\varepsilon)\|_{L^2(\mathbb R^3)}^{(\frac{1}{2}+\frac{3}{r})q}\|(\nabla^3 u_\varepsilon, \nabla^2 v_\varepsilon)\|_{L^2(\mathbb R^3)}^{(\frac{1}{2}-\frac{3}{r})q}dt\right)^{1/q} \\
=&C\left(\int_0^{T^*}\|(\nabla^2 u_\varepsilon, \nabla v_\varepsilon)\|_{L^2(\mathbb R^3)}^{\frac{3q}{2}-2}\|(\nabla^3 u_\varepsilon, \nabla^2 v_\varepsilon)\|_{L^2(\mathbb R^3)}^{2-\frac{q}{2}}dt\right)^{1/q}\\
\leq&C\sup_{0\leq t\leq T^*}\|(\nabla^2 u_\varepsilon, \nabla v_\varepsilon)\|_{L^2(\mathbb R^3)}^{\frac{3}{2}-\frac{2}{q}}\left(\int_0^{T^*}\|(\nabla^3 u_\varepsilon, \nabla^2 v_\varepsilon)\|_{L^2(\mathbb R^3)}^{2-\frac{q}{2}}dt\right)^{1/q}\\
\leq&CM^{\frac{3}{2}-\frac{2}{q}}{T^*}^{\frac{1}{4}}\left(\int_0^{T^*}\|(\nabla^3 u_\varepsilon, \nabla^2 v_\varepsilon)\|_{L^2(\mathbb R^3)}^2dt\right)^{\frac{1}{q}-\frac{1}{4}}\\
\leq&CM^{\frac{3}{2}-\frac{2}{q}}{T^*}^{\frac{1}{4}}M^{\frac{1}{q}-\frac{1}{4}}=CM{T^*}^{\frac{1}{4}}
\end{align*}
for $r\in(6,\infty]$. Due to the above two inequalities, we have
$$
\varlimsup_{\varepsilon\rightarrow0}\|(\nabla u_\varepsilon, v_\varepsilon)\|_{L^q(0, T^*; L^r(\mathbb R^3))}<\infty,
$$
contradicting to (\ref{7.5}). This contradiction implies that $T^*$ must be the maximal existence time, completing the proof.
\end{proof}

\begin{acknowledgement}
{The research of the third author is supported
partially by Zheng Ge Ru Foundation, Hong Kong RGC Earmarked
Research Grant CUHK4041/11P, a Focus Area Grant from The
Chinese University of Hong Kong, and a grant from Croucher Foundation. A part of the work was done when
the first author visited the Chinese University of Hong Kong in Dec of 2012
 and in Jan of 2013.}
\end{acknowledgement}

\end{document}